\documentclass[12pt]{amsart}
\usepackage{amsmath}
\usepackage{amsfonts}
\usepackage{bbm}
\usepackage[all]{xy}
\usepackage{bbding}
\usepackage{txfonts}
\usepackage{amscd}
\usepackage{mathrsfs}

\usepackage{xspace}
\usepackage[T1]{fontenc}
\usepackage[shortlabels]{enumitem}
\allowdisplaybreaks[4]
\usepackage{ifpdf}
\ifpdf
\usepackage[colorlinks,final,backref=page,hyperindex]{hyperref}
\else
\usepackage[colorlinks,final,backref=page,hyperindex]{hyperref}
\fi
\usepackage{tikz}
\usepackage[active]{srcltx}

\newtheorem{thm}{Theorem}[section]

\newtheorem{cor}[thm]{Corollary}
\newtheorem{pro}[thm]{Proposition}
\newtheorem{ex}[thm]{Example}
\newtheorem{rmk}[thm]{Remark}
\newtheorem{defi}[thm]{Definition}

\newcommand {\emptycomment}[1]{}

\newcommand{\be }{\begin{equation}}
	\newcommand{\ee }{\end{equation}}

\newcommand{\Nat}{\mathbb Z_+}

\newcommand{\huaO}{\mathcal{O}}

\newcommand{\g}{\mathfrak g}

\newcommand{\frkB}{\mathfrak B}

\newcommand{\Id}{\rm{Id}}

\newcommand{\br}[1]{   [ \cdot,    \cdot  ]   }
\newcommand{\id}{\mathrm{id}}

\newcommand{\Hom}{\mathrm{Hom}}

\newcommand{\gl}{\mathfrak {gl}}

\newcommand{\ad}{\mathrm{ad}}

\begin{document}
	\title[Hyper relative differential operators on Lie algebras]{Hyper relative differential operators on Lie algebras}

	\author{Sofiane Bouarroudj}
\address{Division of Science and Mathematics, New York University Abu Dhabi, P.O. Box 129188, Abu Dhabi, United Arab Emirates.}
\email{sofiane.bouarroudj@nyu.edu}

\author{Jiefeng Liu}
	\address{School of Mathematics and Statistics, Northeast Normal University, Changchun 130024, China}
	\email{liujf534@nenu.edu.cn}
	
\author{Liwen Zhang}
	\address{School of Mathematics and Statistics, Northeast Normal University, Changchun 130024, China}
	\email{zhangliwen100@nenu.edu.cn}
	
	\begin{abstract}
	In this paper, we first introduce the notion of a hyper relative differential operator on a Lie algebra, in which  Nijenhuis operators are used to characterize the relative differential operators and their inverse. We then introduce the notions of DN-structures,  KN-structures, and KD-structures on Lie algebras and study the relationships between DN-structures, KD-structures, KN-structures, and hyper relative differential operators. Finally, we investigate hyper symplectic structures and hyper Hessian structures from the view point of hyper relative differential operators, and provide equivalent descriptions for both hyper symplectic structures and hyper Hessian structures.
	\end{abstract}
	
	

	\maketitle
	
	\tableofcontents
	\section{Introduction}
Hyper K\"ahler structures on manifolds first appeared in Berger's classification of holonomy groups of Riemannian manifolds, and they are closely related to Yau's solution to the Calabi conjecture, gauge theory, and Rozansky–Witten theory. For further details and references on hyper K\"ahler structures, see \cite{Hit2}. In pursuit of investigating hyper K\"ahler structures on manifolds from a symplectic geometric perspective, Xu naturally introduced the concept of hyper symplectic structures in \cite{Xu}. More precisely, a hyper K\"ahler structure on a manifold $M$ consists of three symplectic structures $\omega_1$, $\omega_2$, and $\omega_3$ satisfying $\big((\omega^\natural_j)^{-1}\circ \omega_i^{\natural}\big)^2=-\Id$ for $i\neq j$. Analogous to hyper K\"ahler structures in manifolds, Hitchin introduced the notion of a para-hyper K\"ahler structure (also called a hyper symplectic structure) on manifold in \cite{Hit}. Due to its important applications in string theory, integrable systems, and differential geometry, para-hyper K\"ahler structures have become an important object in mathematics and mathematical physics. For further details and references on para-hyper K\"ahler structures, see \cite{DS}. Similar to the definition of hyper K\"ahler structures on manifolds given by Xu, a para-hyper K\"ahler structure on a manifold $M$ consists of three symplectic structures $\omega_1$, $\omega_2$, and $\omega_3$ satisfying $\big((\omega^\natural_3)^{-1}\circ \omega_2^{\natural}\big)^2=\Id$, $\big((\omega^\natural_1)^{-1}\circ \omega_3^{\natural}\big)^2=\Id$ and $\big((\omega^\natural_2)^{-1}\circ \omega_1^{\natural}\big)^2=-\Id$. The authors in  \cite{AN13} introduced and studied hyper symplectic structures on Lie algebroids, including hyper K\"ahler structures and para-hyper K\"ahler structures as special cases. If the manifold is a Lie group $G$ and the symplectic structures are left-invariant, hyper K\"ahler structures and para-hyper K\"ahler structures are completely determined by their corresponding hyper K\"ahler structures and para-hyper K\"ahler structures on the Lie algebra $\g$ of $G$.  See \cite{Andrada05,Andrada06,BB,BDF,BS} for more details on hyper K\"ahler Lie algebras and  para-hyper K\"ahler Lie algebras.

Let $\mathfrak{g}$ be a Lie algebra and let $(V;\rho)$ be a representation. A linear map $d:\mathfrak{g} \rightarrow V$ is called a relative differential operator if it satisfies
$$d[x,y]=\rho(x)d(y)-\rho(y)d(x),\quad \text{ for all } x,y\in \g.$$ When the representation is the adjoint representation, a relative differential operator reduces to an ordinary derivation on the Lie algebra. See \cite{CGKZ,GK08,GLSZ,JS} for more details on relative differential operators on Lie algebras and associative algebras. Relative differential operators are closely related to symplectic structures on Lie algebras. More precisely, if $\omega$ is a symplectic structure on $\mathfrak{g}$, then $\omega^\natural$ is a relative differential operator with respect to the coadjoint representation $(\mathfrak{g}^\ast; \operatorname{ad}^\ast)$, defined by $\langle \omega^\natural(x), y \rangle = \omega(x, y)$ for $x, y \in \mathfrak{g}$. Since the definitions of hyper Kähler Lie algebras and para-hyper Kähler Lie algebras rely fundamentally on compositions of three symplectic structures, it is natural to extend these constructions to the setting of relative differential operators. To this end, we introduce the notion of an $\varepsilon$-hyper differential operator on a Lie algebra, which consists of three invertible relative differential operators $d_1, d_2, d_3$ such that 
$$\big(T_{i-1}\circ d_{i+1})^2=\varepsilon_i\Id_\g,\quad \varepsilon_i=\pm 1,$$
where $T_i=d^{-1}_i$ for $i=1,2,3$. In particular, a hyper Kähler Lie algebra gives rise to an $\varepsilon$-hyper differential operator with $\varepsilon=(-1,-1,-1)$, and a para-hyper Kähler Lie algebra gives rise to  an $\varepsilon$-hyper differential operator with $\varepsilon=(1,1,-1)$ (see Proposition \ref{pro:hyper-kahler}). 

Let $(d_1,d_2,d_3)$ be an $\varepsilon$-hyper relative differential operator on a Lie algebra $\g$. Then $N_i:=T_{i-1}\circ d_{i+1}$, for $i\in\mathbb{Z}_3$, are Nijenhuis operators on the Lie algebra $\g$ (see Proposition \ref{pro:Nij}). We refer to  \cite{Dorfman} for more details on Nijenhuis operators and their applications. In particular, if $\varepsilon_{i}=-1$, then $N_i$ is a complex structure on $\g$, while if $\varepsilon_{i}=1$, then $N_i$ is a para-complex structure on $\g$. For more details on complex and para-complex structures on Lie algebras, see \cite{Andrada}.
 
Relative differential operators are also closely related to $\huaO$-operators on Lie algebras. If a relative differential operator is invertible, then its inverse is an $\huaO$-operator. The notion of an $\huaO$-operator was introduced by Kupershmidt in \cite{Kuper1} as a tool to better understand the interrelation between the classical Yang-Baxter equation and related integrable systems. In particular, a Rota-Baxter operator as given by Semonov-Tian-Shansky (\cite{STS}) is precisely an $\huaO$-operator on a Lie algebra $\g$ with respect to the adjoint representation $(\g;\ad)$, while a skew-symmetric solution of the classical Yang-Baxter equation is an $\huaO$-operator with respect to the coadjoint representation $(\g^*;\ad^*)$. Analogous to the Poisson–Nijenhuis structures in \cite{Kosmann1}, the authors in \cite{HLS} introduced the notion of a Kupershmidt-dual-Nijenhuis structure (KN-structure), which consists of an $\huaO$-operator and a dual-Nijenhuis pair satisfying certain compatibility conditions. In this paper, by adding compatibility conditions between a relative differential operator and an $\huaO$-operator, we introduce the notion of a KD-structure on a Lie algebra, and by adding compatibility conditions between a relative differential operator and a Nijenhuis operator, we introduce the notion of a DN-structure on a Lie algebra. The relationships among KD-structures, DN-structures, and KN-structures are studied thoroughly. More notably, these structures can be constructed using hyper relative differential operators.

Another interesting example of relative differential operators arises from Hessian structures on pre-Lie algebras. A Hessian structure on a pre-Lie algebra corresponds to an affine Lie group $G$ equipped with a $G$-invariant Hessian structure. Hessian structures on manifolds are related to K\"ahlerian geometry, symplectic geometry, affine differential geometry, and information geometry (\cite{Shima07}). For example, the tangent bundle of a Hessian manifold naturally produces a K\"ahlerian structure, and many important families of probability distributions, such as exponential families, admit Hessian structures. Recall that a Hessian structure on a pre-Lie algebra $(\g, \cdot)$ is a nondegenerate symmetric bilinear form $\mathscr{B}:\g\otimes\g\rightarrow \mathbb{R}$ satisfying
$$\mathscr{B}(x\cdot y,z)-\mathscr{B}(x,y\cdot z)=\mathscr{B}(y\cdot x,z)-\mathscr{B}(y,x\cdot z),\quad \text{ for all } x,y,z\in\g.$$
Furthermore, the nondegenerate symmetric bilinear form $\mathscr{B}$ defines a Hessian structure on the pre-Lie algebra $\g$ if and only if $\mathscr{B}^\natural$ is a relative differential operator on the Lie algebra $\g^c$ with respect to the coregular representation $(\g^*;L^*)$, where $\mathscr{B}^\natural:\g\rightarrow\g^*$ is given by $\langle\mathscr{B}^\natural(x),y\rangle=\mathscr{B}(x,y)$ for $x,y\in\g$ and the coregular representation $L^*$ of $\g^c$ on $\g^*$ is given by $\langle L^*_x \xi,y\rangle=-\langle \xi,x\cdot y\rangle$ for $\xi\in\g^*$.  As an application of relative differential operators on Lie algebras, we introduce the notion of a hyper Hessian structure on a pre-Lie algebra $\g$, which consists of three Hessian structures $\frkB_1$, $\frkB_2$ and $\frkB_3$ such that $(\frkB_1^\natural,\frkB_2^\natural,\frkB_3^\natural)$ is a hyper relative differential operator on the Lie algebra $\g^c$ with respect to the coregular representation. Moreover, the relationships between hyper anti-K\"ahler structures, para-hyper anti-K\"ahler structures, and hyper Hessian structures are well studied.

The paper is organized as follows. In Section \ref{sec:prel}, we introduce the notion of a hyper differential operator on a Lie algebra and show that it gives rise to three Nijenhuis operators on the Lie algebra via compositions of relative differential operators and their inverses. We also analyze the basic properties of hyper differential operators. In Section \ref{sec:compatible}, we introduce the notions of DN-structures and KD-structures on Lie algebras, and show that hyper relative differential operators naturally induce such structures. We then recall KN-structures on Lie algebras and study the relationships among DN-structures, KD-structures, and KN-structures. In particular, we prove that hyper relative differential operators also give rise to KN-structures. In Section \ref{sec:2cases}, we study two kinds of $\varepsilon$-hyper relative differential operators, namely those with $\varepsilon_1\varepsilon_2\varepsilon_3=1$ and those with $\varepsilon_1\varepsilon_2\varepsilon_3=-1$. More compatible structures and equivalent descriptions of hyper relative differential operators are provided. In Section \ref{sec:hyper2}, we apply hyper differential operators on Lie algebras to study hyper symplectic structures and hyper Hessian structures. In particular, we show that a hyper symplectic structure is equivalent to a hyper differential operator on the Lie algebra for which the differential operators are skew-symmetric, while a hyper Hessian structure is equivalent to a hyper differential operator on the pre-Lie algebra for which the differential operators are symmetric.
	
	\section{Hyper relative differential operators on Lie algebras} \label{sec:prel}
In this section, we first recall the definition of a relative differential operator on a Lie algebra, defined as a $1$-cocycle on the Lie algebra with values in a representation. We then  recall the relationships  between relative differential operators, $\huaO$-operators, symplectic structures and Hessian structures. Finally, we introduce and develop the notion of \emph{hyper relative differential operators} and analyze their properties. 
\begin{defi}
Let $(\g,[\cdot,\cdot])$ be a Lie algebra and $\rho:\g\rightarrow \gl(V)$ be a representation of $\g$ on $V$. A linear map $d:\g\rightarrow V$ is called a \emph{relative differential operator} if $d$ satisfies
\begin{equation}\label{eq:RDO}
d[x,y]=\rho(x)d(y)-\rho(y)d(x),\quad \text{ for all } x,y\in \g.
\end{equation}
\end{defi}
This operator is the usual derivation on $\g$ in the case where $\rho$ is the adjoint representation. Moreover, since the $1$-st cohomology group of a  semisimple Lie algebra $\g$ is trivial, the set of relative differential operators on $\g$ is given by $$\{d\in \Hom(\g,V)\mid \exists~ u\in V\, \text{ with }\; d =\rho(\cdot )\, u\}.$$ 
There are no invertible relative differential operators on semisimple Lie algebras.

Let $(\g,[\cdot,\cdot])$ be a Lie algebra and $\rho:\g\rightarrow \gl(V)$ be a representation of $\g$ on $V$. Following \cite{Kuper1}, an \emph{$\huaO$-operator} is a linear map $T:V\rightarrow\g$ satisfying
\begin{equation}\label{eq:RRB}
[Tu,Tv]=T(\rho(Tu)v-\rho(Tv)u),\quad \text{ for all } u,v\in V.
\end{equation}

The following proposition demonstrates the relationship between relative differential operators and $\huaO$-operators.
\begin{pro}
Let $d:\g\rightarrow V$ be an invertible linear map. Then $d:\g\rightarrow V$ is a relative differential operator with respect to the representation $(V;\rho)$ if and only if $d^{-1}:V\rightarrow \g$ is an $\huaO$-operator with respect to the representation $(V;\rho)$.
\end{pro}

Recall that a \emph{symplectic structure}\footnote{This structure is sometimes  refer to as quasi-Frobenuis, see \cite{BM}} on a Lie algebra $\g$ is a nondegenerate $2$-form $\omega\in \wedge^2\g^*$ satisfying
\begin{equation}
\omega([x,y],z)+\omega([z,x],y)+\omega([y,z],x)=0,\quad \text{ for all } \, x,y,z\in \g.
\end{equation}
This means that $\omega$ is a 2-cocycle on $\g$ with values in $\mathbb K$. Such structures exist only in even dimensions. For classifications in low dimensions, see \cite{MR2}. 
\begin{ex}{\rm
Let $\omega$ be a nondegenerate $2$-form on a Lie algebra $\g$. Define $\omega^\natural:\g\rightarrow\g^*$ by 
\begin{equation}\label{eq:sym-diff}
\langle\omega^\natural(x),y\rangle=\omega(x,y),\quad x,y\in \g.
\end{equation}
Then $\omega$ is a symplectic structure on the Lie algebra $\g$ if and only if $\omega^\natural$ is a relative differential operator with respect to the coadjoint representation $(\g^*;\ad^*)$, or equivalently,  $(\omega^\natural)^{-1}$ is an $\huaO$-operator with respect to $(\g^*;\ad^*)$.}
\end{ex}
Recall that a \emph{pre-Lie algebra} is a vector space $\g$ equipped with a binary operation $\cdot:\g\otimes\g\rightarrow\g$ satisfying
\begin{equation}
(x\cdot y)\cdot z-x\cdot (y\cdot z)=(y\cdot x)\cdot z-y\cdot(x\cdot z),\quad \text{ for all } x,y,z\in\g.
\end{equation}
Pre-Lie algebras are also referred to as {\it left-symmetric} algebras in the literature; see \cite{Bai21,Bur3} and references therein.

Let $(\g,\cdot)$ be a pre-Lie algebra. The commutator $[x,y]_\g=x\cdot y-y\cdot x$ defines a Lie algebra $\g^c$, called the \emph{ sub-adjacent Lie algebra} of
$(\g,\cdot)$. In this case, $\g$ is called a \emph{compatible pre-Lie algebra structure} on the Lie algebra $\g^c$. Furthermore, the map $L:\g\longrightarrow \gl(\g)$ with $x\mapsto  L_x$, where $L_xy=x\cdot y$, for all $x,y\in \g$, defines a representation of the Lie algebra $\g^c$ on $\g$. This representation is called a  regular representation of $\g^c$, and its dual representation $(\g^*; L^*)$ is called a coregular representation. For further details, see \cite{Bai21} and the references therein.

Following \cite{Shima07}, a \emph{Hessian structure} on a pre-Lie algebra $(\g, \cdot)$ is a nondegenerate symmetric bilinear form $\mathscr{B}:\g\otimes\g\rightarrow \mathbb{R}$ satisfying
\begin{equation}
 \mathscr{B}(x\cdot y,z)-\mathscr{B}(x,y\cdot z)=\mathscr{B}(y\cdot x,z)-\mathscr{B}(y,x\cdot z),\quad \text{ for all } x,y,z\in\g.
\end{equation}
\begin{ex} {\rm We present several examples of Hessian structures on Lie algebras. The Lie algebras themselves are taken from \cite{Bur}, whereas the Hessian structures are computed here.

(i) Two-dimensional pre-Lie algebras were classified in \cite{Bur}. None of them admits a Hessian structure.

(ii) The pre-Lie algebra $I_4$ is defined, with respect to the basis $\{e_1, e_2, e_3, e_4\}$, by: 
\[
e_1\cdot e_1= 2 \, e_1,\quad  e_1\cdot e_i=e_i,\,  \text{ for } i=1,\cdots,4, \quad e_j\cdot e_j=e_1, \, \text{ for } j=2,3,4.
\]
The Hessian structure is given by
\[
\mathscr{B}= 2 e_1^* \otimes e_1^* + e_2^* \otimes e_2^* + e_3^* \otimes e_3^*+ e_4^* \otimes e_4^*.
\]
(iii) The pre-Lie algebra $B_4$  defined, with respect to the basis $\{e_1, e_2, e_3, e_4\}$, by 
\[
e_1\cdot e_2= e_2 \cdot e_1=e_4,\;\;  e_2 \cdot e_3=2 \, e_1, \;\;   e_3\cdot e_2=e_1, \;\; e_4 \cdot e_2=-e_2, \;\;  e_4\cdot e_3=e_3,\;\;  e_4 \cdot e_4=-e_4.
\]
doesn't have a Hessian structure.

(iv) The pre-algebra $A_4$ is defined, with respect to the basis $\{e_1, e_2, e_3, e_4\}$, by 
\[
e_1\cdot e_i=\alpha_i \, e_i,\quad e_j\cdot e_{6-j}=e_1, \quad \text{for } \, i=1,\ldots,4\,  \text{ and } j=2,3,4,
\]
where $\alpha_3=\frac{1}{2}\alpha_1, \, \alpha_4=\alpha_1-\alpha_2$ are constant. The Hessian structure is given by
\[
\mathscr{B}= \alpha_1 e_1^* \otimes e_1^* + e_2^*\otimes e_4^* + e_4^*\otimes e_2^* + e_3^* \otimes e_3^*.
\]
}
\end{ex}
\begin{pro}
Let $\mathscr{B}$ be a nondegenerate symmetric bilinear form on a pre-Lie algebra $\g$. Define $\mathscr{B}^\natural:\g\rightarrow\g^*$ by 
\begin{equation}\label{eq:hess-diff}
\langle\mathscr{B}^\natural(x),y\rangle=\mathscr{B}(x,y),\quad \text{ for every } \, x,y\in \g.
\end{equation}
Then $\mathscr{B}$ is a Hessian structure on the pre-Lie algebra $\g$ if and only if  $\mathscr{B}^\natural$ is a relative differential operator on the Lie algebra $\g^c$ with respect to the coregular representation $(\g^*;L^*)$, or equivalently, $(\mathscr{B}^\natural)^{-1}$ is an $\huaO$-operator on the Lie algebra $\g^c$ with respect to $(\g^*;L^*)$.
\end{pro}
\begin{proof}
By a direct calculation, we have
\begin{eqnarray*}
&&\mathscr{B}(x\cdot y,z)-\mathscr{B}(x,y\cdot z)-\mathscr{B}(y\cdot x,z)+\mathscr{B}(y,x\cdot z)\\
&=&\langle\mathscr{B}^\natural([x,y]),z\rangle-\langle\mathscr{B}^\natural(x),y\cdot z\rangle+\langle\mathscr{B}^\natural(y),x\cdot z\rangle\\
&=&\langle\mathscr{B}^\natural([x,y])-L_x^*\mathscr{B}^\natural(y)+L_y^*\mathscr{B}^\natural(x),z\rangle,
\end{eqnarray*}
which implies that $\mathscr{B}$ is a Hessian structure on the pre-Lie algebra if and only if  $\mathscr{B}^\natural$ is a relative differential operator with respect to the coregular representation $(\g^*;L^*)$.
\end{proof}
Now, we are ready to introduce our main definition. 
\begin{defi}
Let $d_1,d_2$ and $d_3$ be invertible relative differential operators with respect to the representation $(V;\rho)$. Set $T_i:=d_i^{-1}$ and $N_i:=T_{i-1}\circ d_{i+1}$, for $i\in\mathbb{Z}_3$. If $N_i$ satisfies $N_i^2=\varepsilon_i\Id_\g$ with $\varepsilon_i=\pm 1$ for $i=1,2,3$, then the triple
$(d_1,d_2,d_3)$ is called an \emph{$\varepsilon=(\varepsilon_1,\varepsilon_2,\varepsilon_3)$-hyper relative differential operator}. 
\end{defi}


\begin{ex}{\rm  Consider the Lie algebra $L$ given in the basis $\{e_1,\ldots,e_4\}$ by
 \[
 [e_1,e_2]=e_2 , \quad [e_1, e_3]=-e_3, \quad [e_1, e_4]=e_4.
 \]
 There are three symplectic structures on $L$ given by 
 \[
 \begin{array}{lcl}
 \omega_1 & = & e_1^*\wedge e_2^*+ e_1^*\wedge e_3^*+ e_3^*\wedge e_4^*,\\[2mm]
 \omega_2 & = &  e_1^*\wedge e_4^*+ e_2^*\wedge e_3^*,\\[2mm]
 \omega_3 & = & -e_1^*\wedge e_2^*+ e_1^*\wedge e_3^*+  e_3^*\wedge e_4^*.
 \end{array}
 \]
The triple $(\omega^\natural_1, \omega^\natural_2, \omega^\natural_3)$ yields a $(1,1,-1)$-hyper relative differential operator.
}
\end{ex}

Following \cite{Dorfman},  a \emph{Nijenhuis operator} on a Lie algebra $\g$ is a linear map $N:\g\longrightarrow\g$ satisfying
\begin{equation}\label{eq:Nij}
  [N(x),N(y)]=N\big([N(x),y]+[x,N(y)]-N[x,y]\big),\quad \text{ for all } x,y\in\g.
\end{equation}
Furthermore, the deformed bracket $[-,-]_N:\wedge^2\g\longrightarrow\g$ given by
\begin{equation}\label{eq:deformbracket}
  [x,y]_N:=[N(x),y]+[x,N(y)]-N([x,y]),
\end{equation}
 is a Lie bracket and $N$ is a Lie algebra morphism from $(\g,[\cdot,\cdot]_N)$ to $(\g,[\cdot,\cdot])$.

In particular, the Nijenhuis operator $N$ satisfying $N^2=-\Id_\g$  is called a \emph{complex  structure}, while the Nijenhuis operator $N$ satisfying $N^2=\Id_\g$  is called a \emph{para-complex  structure}.

In \cite{Win}, the author showed that the pair $(\g,\cdot)$ is a real pre-Lie algebra if and only if the right multiplication on $\g$ is a linear Nijenhuis operator on the real affine space $\g$ (it is not a Nijenhuis operator on the sub-adjacent Lie algebra of $\g$). 


\begin{pro}\label{pro:Nij}
Let $(d_1,d_2,d_3)$ be an $\varepsilon$-hyper relative differential operator on the Lie algebra $\g$. For any $i\in\mathbb{Z}_3$, the operator $N_i:=T_{i-1}\circ d_{i+1}$,  is a Nijenhuis operator on $\g$. In particular, if $\varepsilon_{i}=-1$, then $N_i$ is a complex structure on  $\g$, while if $\varepsilon_{i}=1$, then $N_i$ is a para-complex structure on $\g$.
\end{pro}
\begin{proof}
By the fact that $d_i$ is a relative differential operator and $T_{i-1}$ is an $\huaO$-operator, the left-hand side of Eq. \eqref{eq:Nij} is equal to
\begin{align*}
[N_ix,N_iy]&=[T_{i-1}(d_{i+1}x),T_{i-1}(d_{i+1}y)]\\
&=T_{i-1}\Big(\rho\big(T_{i-1}(d_{i+1}x)\big)(d_{i+1}y)-\rho\big(T_{i-1}(d_{i+1}y)\big)(d_{i+1}x)\Big).
\end{align*}
Furthermore, since $N_i^2=\varepsilon_i\Id_\g$, we have $d_{i+1}\circ T_{i-1}\circ d_{i+1}=d_{i-1}\varepsilon_i$. Combining with the fact that $d_{i+1}$ is a relative differential operator, we have
\begin{align*}
  N_i([N_ix,y]) & =T_{i-1}\big(d_{i+1}[T_{i-1}(d_{i+1}x),y]\big) \\
  &=T_{i-1}\Big(\rho\big(T_{i-1}(d_{i+1}x)\big)(d_{i+1}y) -\rho(y)\big((d_{i+1}\circ T_{i-1}\circ d_{i+1})(x)\big)\Big) \\
  &=T_{i-1}\Big(\rho\big(T_{i-1}(d_{i+1}x)\big)(d_{i+1}y) -\varepsilon_i\rho(y)\big(d_{i-1}x\big)\Big).
\end{align*}
Similarly, we have
$$ N_i([x,N_iy])=T_{i-1}\Big(\varepsilon_i\rho(x)\big(d_{i-1}y\big)-\rho\big(T_{i-1}(d_{i+1}y)\big)(d_{i+1}x) \Big).$$
By the fact that $d_{i-1}$ is a relative differential operator and $T_{i-1}\circ d_{i-1}=\Id_\g$, the right-hand side of Eq. \eqref{eq:Nij} is equal to
\begin{align*}
&N_i\big([N_ix,y]+[x,N_iy]-N_i[x,y]\big)\\
&=T_{i-1}\Big(\rho\big(T_{i-1}(d_{i+1}x)\big)(d_{i+1}y) -\varepsilon_i\rho(y)\big(d_{i-1}x\big)\Big)\\
&\quad+T_{i-1}\Big(\varepsilon_i\rho(x)\big(d_{i-1}y\big)-\rho\big(T_{i-1}(d_{i+1}y)\big)(d_{i+1}x) \Big)-\varepsilon_i[x,y]\\
&=T_{i-1}\Big(\rho\big(T_{i-1}(d_{i+1}x)\big)(d_{i+1}y)-\rho\big(T_{i-1}(d_{i+1}y)\big)(d_{i+1}x)\Big)\\
&\quad+\varepsilon_iT_{i-1}\Big(\rho(x)\big(d_{i-1}y\big)-\rho(y)\big(d_{i-1}x\big)\Big)-\varepsilon_i[x,y]\\
&=T_{i-1}\Big(\rho\big(T_{i-1}(d_{i+1}x)\big)(d_{i+1}y)-\rho\big(T_{i-1}(d_{i+1}y)\big)(d_{i+1}x)\Big)+\varepsilon_iT_{i-1}(d_{i-1}[x,y])-\varepsilon_i[x,y]\\
&=[N_ix,N_iy].
\end{align*}
Thus, $N_i$ is a Nijenhuis operator on the Lie algebra $\g$.
\end{proof}


Define $h^\flat:\g\rightarrow V$ by
\begin{equation}\label{eq:hyper-metric}
h^\flat=\varepsilon_{3}\varepsilon_{2} \, d_{3}\circ T_{1}\circ d_{2}.
\end{equation}
This map will be useful in describing hyper relative differential operators in Section \ref{sec:2cases}. 
\begin{pro}\label{pro:important property}
Let $(d_1,d_2,d_3)$ be an $\varepsilon$-hyper relative differential operator on a Lie algebra $\g$. Then, $h^\flat$ satisfies
\begin{equation}\label{eq:hproperty1}
h^\flat=\varepsilon_{i-1}\varepsilon_{i+1} \, d_{i-1}\circ T_{i}\circ d_{i+1}.
\end{equation}
Furthermore, the relations between $h^\flat$, $d_i$ and $N_i$ are given by
\begin{align}
\label{eq:hdnproperty1} T_i\circ S_i&=\varepsilon_{1}\varepsilon_{2}\varepsilon_{3}\, N_i\circ T_i=\varepsilon_{i+1}(h^\flat)^{-1},\\
\label{eq:hdnproperty2} d_i\circ N_i&=\varepsilon_{1}\varepsilon_{2}\varepsilon_{3}\, S_i\circ d_i=\varepsilon_{i-1}h^\flat,\\
\label{eq:hdnproperty3}  h^\flat\circ N_i&=\varepsilon_{1}\varepsilon_{2}\varepsilon_{3}\, S_i\circ h^\flat=\varepsilon_{i}\varepsilon_{i-1}\, d_i,
\end{align}
where $S_i=d_{i+1}\circ T_{i-1}$ for $i\in \mathbb{Z}_{3}$.
\end{pro}
\begin{proof}
By the definition of $N_i$ and $N_i^2=\varepsilon_i\Id_\g$, we have
\begin{align}
\label{eq:dtproperty1} d_{i+1}\circ T_{i-1}=\varepsilon_id_{i-1}\circ T_{i+1},\\
\label{eq:dtproperty2} T_{i+1}\circ d_{i-1}=\varepsilon_iT_{i-1}\circ d_{i+1}.
\end{align}
By Eq. \eqref{eq:dtproperty1} and Eq. \eqref{eq:dtproperty2}, we obtain 
\begin{align*}
h^\flat&=\varepsilon_3\varepsilon_2 \, d_3\circ T_1\circ d_2=\varepsilon_2 \, d_3\circ T_2\circ d_1=\varepsilon_2\varepsilon_1 \, d_2\circ T_3\circ d_1\\
&=\varepsilon_1 \, d_2\circ T_1\circ d_3=\varepsilon_1\varepsilon_3d_1\circ T_2\circ d_3,
\end{align*}
which implies that  $h^\flat=\varepsilon_{i-1}\varepsilon_{i+1} \, d_{i-1}\circ T_{i}\circ d_{i+1}$ for any $i\in\mathbb{Z}_3$.

Furthermore, by Eq. \eqref{eq:dtproperty1}, we have
\begin{align*}
  \varepsilon_{i-1}h^\flat&=\varepsilon_{i+1}d_{i-1}\circ T_{i}\circ d_{i+1}=d_i\circ T_{i-1}\circ d_{i+1}=d_{i}\circ N_i.
\end{align*}
By Eq. \eqref{eq:hproperty1}, we have
\begin{align*}
(h^\flat)^{-1}&=(\varepsilon_{i-1}\varepsilon_{i+1}d_{i-1}\circ T_i\circ d_{i+1})^{-1}=\varepsilon_{i-1}\varepsilon_{i+1}T_{i+1}\circ d_i\circ T_{i-1}\\
&=\varepsilon_{i+1}\varepsilon_{i}T_i\circ d_{i-1}\circ T_{i+1}=\varepsilon_{i}\varepsilon_{i-1}T_{i-1}\circ d_{i+1}\circ T_{i}=\varepsilon_{i}\varepsilon_{i-1}N_i\circ T_i.
\end{align*}
By Eq. \eqref{eq:dtproperty1}, we have
\begin{align*}
T_i\circ S_i&=T_i\circ d_{i+1}\circ T_{i-1}=\varepsilon_iT_i\circ d_{i-1}\circ T_{i+1}\\
&=(\varepsilon_{i+1})^2\varepsilon_iT_i\circ d_{i-1}\circ T_{i+1}=\varepsilon_{i+1}(h^\flat)^{-1}\\
&=\varepsilon_{i+1}\varepsilon_{i}\varepsilon_{i-1}N_i\circ T_i=\varepsilon_1\varepsilon_2\varepsilon_3N_i\circ T_i.
\end{align*}
This proves Eq.  \eqref{eq:hdnproperty1}. The proofs of Eq.\eqref{eq:hdnproperty2} and Eq.\eqref{eq:hdnproperty3} are analogous.
\end{proof}
The following properties of hyper relative differential operators play an important role in the study of their relationships with KN-structures, KD-structures, and DN-structures in the subsequent sections.
\begin{pro}\label{pro:pro2.8}
Let $(d_1,d_2,d_3)$ be an $\varepsilon$-hyper relative differential operator on a Lie algebra $\g$. For all $i,j\in\mathbb{Z}_{3}$ such that $i\neq j$, we have
\begin{itemize}
\item[\rm(i)] $$T_i\circ S_j=N_j\circ T_i=\begin{cases}\varepsilon_{i+1}\, T_{i-1},&j=i+1,\\
T_{i+1},&j=i-1.\end{cases}$$
\item[\rm(ii)] $$d_i\circ N_j=S_j\circ d_i=\begin{cases}d_{i-1},&j=i+1,\\
\varepsilon_{i-1} \, d_{i+1},&j=i-1.\end{cases}$$
\item[\rm(iii)] $$N_i\circ N_j=\varepsilon_{1}\varepsilon_{2}\varepsilon_{3}\, N_j\circ N_i=\begin{cases}\varepsilon_{i}\varepsilon_{i+1}\, N_{i-1},&j=i+1,\\
\varepsilon_{i+1}N_{i+1},&j=i-1.\end{cases}$$
\item[\rm(iv)] $$S_i\circ S_j=\varepsilon_{1}\varepsilon_{2}\varepsilon_{3}\, S_j\circ S_i=\begin{cases}\varepsilon_{i}\varepsilon_{i+1}\, S_{i-1},&j=i+1,\\
\varepsilon_{i+1}\, S_{i+1},&j=i-1,\end{cases}$$
and $S_i^2=\varepsilon_i \, \Id_V$.
\item[\rm(v)] $$T_k\circ d_i=\begin{cases}\Id_\g,&k=i,\\
N_{i-1},&k=i+1,\\
\varepsilon_{i+1}\, N_{i+1},&k=i-1.\end{cases}$$
\item[\rm(vi)] $$d_i\circ T_k=\begin{cases}\Id_V,&k=i,\\
S_{i-1},&k=i+1,\\
\varepsilon_{i+1}\, S_{i+1},&k=i-1.\end{cases}$$
\end{itemize}
\end{pro}
\begin{proof}
(i) For $j=i+1$, we have
$$T_i\circ S_j=T_i\circ S_{i+1}=T_i\circ d_{i-1}\circ T_i=N_{i+1}\circ T_{i}=N_j\circ T_i.$$
By Eq. \eqref{eq:dtproperty1}, we get
$$T_i\circ S_j=N_j\circ T_i=\varepsilon_{i+1}T_i\circ d_i\circ T_{i-1}=\varepsilon_{i+1}T_{i-1}.$$
For $j=i-1$, we have
$$T_i\circ S_j=T_i\circ S_{i-1}=T_i\circ d_i\circ T_{i+1}=T_{i+1}=T_{i+1}\circ d_i\circ T_i=N_{i-1}\circ T_i=N_j\circ T_i.$$

Items (ii) -- (vi) can be proved similarly.
\end{proof}

\section{Hyper relative differential operators and compatible structures on Lie algebras}\label{sec:compatible}
In this section, we first introduce DN-structures and KD-structures on Lie algebras, and show that hyper relative differential operators naturally induce such structures. We then recall the notion of KN-structures on Lie algebras and study the relationships among DN-structures, KD-structures, and KN-structures. In particular, we prove that hyper relative differential operators also give rise to KN-structures.

\subsection{Hyper relative differential operators, {\rm DN}-structures and {\rm KD}-structures on Lie algebras}

In a KD-structure, K refers to an $\huaO$-operator introduced by Kupersdmit, while D stands for a relative differential operator.  Similarly, in a DN-structure, D refers to a relative differential operator, while N refers to a Nijenhuis operator. 
\begin{defi}
Let $d$ be a relative differential operator and $T$ an $\huaO$-operator on a Lie algebra $\g$ with respect to a representation $(V;\rho)$. A pair $(T,d)$ is called a \emph{KD-structure} on  $\g$ with respect to the representation $(V;\rho)$ if $d_N=d\circ N$ is also a relative differential operator on $\g$ with respect to the representation $(V;\rho)$, where $N:=T\circ d$.
\end{defi}

\begin{pro}\label{pro:hyper-KD}
Let $(d_1,d_2,d_3)$ be an $\varepsilon$-hyper relative differential operator on a Lie algebra $\g$. Then $(T_{i},d_k)$  are {\rm KD}-structures for all $i\neq k\in\mathbb{Z}_{3}$.
\end{pro}
\begin{proof}
By \rm(ii) in Prop. \ref{pro:pro2.8} and the fact that $d_{i+1}$ is a relative differential operator, letting $d=d_i$ and $N=T_{i+1}\circ d_i$,  we have
\begin{align*}
\rho(x)d_{N}(y)-\rho(y)d_{N}(x)&=\rho(x)(d_i T_{i+1}d_i)(y)-\rho(y)(d_i T_{i+1} d_i)(x)\\
&=\rho(x)(\varepsilon_{i-1}d_{i+1})(y)-\rho(y)(\varepsilon_{i-1}d_{i+1})(x)\\
&=\varepsilon_{i-1}d_{i+1}[x,y]=d_i T_{i+1} d_i[x,y]\\
&=d_{N}[x,y],
\end{align*}
which implies that $d_N=d\circ N$ is also a relative differential operator and thus $(T_{i+1},d_i)$ is a {\rm KD}-structure. Similarly,  we can prove that $(T_{i-1},d_i)$ is also a {\rm KD}-structure. These facts imply that $(T_{i},d_k)$  are {\rm KD}-structures for all $i\neq k\in\mathbb{Z}_{3}$.
\end{proof}

\begin{defi}
Let $d$ be a relative differential operator on a Lie algebra $\g$ with respect to a  representation $(V;\rho)$ and $N:\g\rightarrow \g$ be a Nijenhuis operator. A pair $(d,N)$ is called a  \emph{DN-structure} on $\g$ with respect to  $(V;\rho)$ if $d_N=d\circ N$ is  a relative differential operator on $\g$ with respect to  $(V;\rho)$.
\end{defi}

\begin{pro}
Let $(d,N)$ be a {\rm DN}-structure on a Lie algebra $\g$ with respect to a representation $(V;\rho)$. Then $d_{N^k}=d\circ N^k$ for all $k\in \mathbb{Z}_+$ are relative differential operators on  $\g$ with respect to  $(V;\rho)$. Furthermore, $(d_{N^k},N^k)$ are {\rm DN}-structures for all $k\in \mathbb{Z}_+$.
\end{pro}
\begin{proof}
For $k=1$, we have $d_{N}=d\circ N$, which is a relative differential operator by definition of a {\rm DN}-structure.

Now we assume that $d_{N^k}=d\circ N^k$ are relative differential operators for all positive integers $1<k\leq m$. We will prove that  $d_{N^{m+1}}=d\circ N^{m+1}$ is a relative differential operator. Since $N$ is a Nijenhuis operator, we have
\begin{align*}
d_{N^{m+1}}[x,y]&=(d\circ N^{m+1})[x,y]=(d\circ N^{m-1})(N^2[x,y])\\
&=(d\circ N^{m-1})\big(N[Nx,y]+N[x,Ny]-[N x,Ny]\big)\\
&=(d\circ N^m)[Nx,y]+(d\circ N^m)[x,Ny]-(d\circ N^{m-1})([Nx,Ny]).
\end{align*}
Since $d\circ N^m$ and $d\circ N^{m-1}$ are relative differential operators by inductive hypothesis, we get
\begin{align*}
&(d\circ N^m)[Nx,y]=\rho(Nx)(d\circ N^m)(y)-\rho(y)(d\circ N^m)(Nx),\\
&(d\circ N^m)[x,Ny]=\rho(x)(d\circ N^m)(Ny)-\rho(Ny)(d\circ N^m)(x),\\
&(d\circ N^{m-1})[Nx,Ny]=\rho(Nx)(d\circ N^{m})(y)-\rho(Ny)(d\circ N^{m})(x).
\end{align*}
Finally, we obtain
\begin{align*}
d_{N^{m+1}}[x,y]&=\rho(Nx)(d\circ N^m)(y)-\rho(y)(d\circ N^m)(Nx)+\rho(x)(d\circ N^m)(Ny)-\rho(Ny)(d\circ N^m)(x)\\
&\quad-\big(\rho(Nx)(d\circ N^{m})(y)-\rho(Ny)(d\circ N^{m})(x)\big)\\
&=\rho(x)(d\circ N^m)(Ny)-\rho(y)(d\circ N^m)(Nx)\\
&=\rho(x)(d\circ N^{m+1})(y)-\rho(y)(d\circ N^{m+1})(x)\\
&=\rho(x)(d_{N^{m+1}})(y)-\rho(y)(d_{N^{m+1}})(x),
\end{align*}
which implies that $d_{N^k}=d\circ N^k$ for all $k\in \mathbb{Z}_+$ are relative differential operators on $\g$ with respect to $(V;\rho)$. Since $N^k$ for all $k\in \mathbb{Z}_+$ are also Nijenhuis operators on  $\g$, it follows that $(d_{N^k},N^k)$ are {\rm DN}-structures on $\g$ with respect to $(V;\rho)$  for all $k\in \mathbb{Z}_+$.
\end{proof}
The following proposition demonstrates the relationships between DN-structres and KD-structures.
\begin{pro}\label{pro:hyper-KD-DN}
If $(d,N)$ is a {\rm DN}-structure on a Lie algebra $\g$ with respect to a representation $(V;\rho)$ and $d$ is invertible, then $(T=N\circ d^{-1},d)$ is a {\rm KD}-structure on $\g$ with respect to  $(V;\rho)$. Conversely, if $(T,d)$ is a {\rm KD}-structure, then $(d,N=T\circ d)$ is a {\rm DN}-structure.
\end{pro}
\begin{proof}
Let $(d,N)$ be a {\rm DN}-structure. Since $N$ is a Nijenhuis operator, we have
\begin{align*}
[Tu,Tv]&=[N\circ d^{-1}u,N\circ d^{-1}v]\\
&=N\big([N\circ d^{-1}u,d^{-1}v]+[d^{-1}u,N\circ d^{-1}v]-N[d^{-1}u,d^{-1}v]\big)\\
&=(T\circ d)\big([Tu,d^{-1}v]+[d^{-1}u,Tv]-N[d^{-1}u,d^{-1}v]\big).
\end{align*}
On the other hand, since $d$ and $d\circ N$ are relative differential operators, we have
\begin{align*}
(T\circ d)[Tu,d^{-1}v]&=T\big(\rho(Tu)v-\rho(d^{-1}v)d(Tu)\big),\\
(T\circ d)[d^{-1}u,Tv]&=T\big(\rho(d^{-1}u)d(Tv)-\rho(Tv)u\big),\\
(T\circ d\circ N)[d^{-1}u,d^{-1}v]&=T\big(\rho(d^{-1}u)d(Tv)-\rho(d^{-1}v)d(Tu)\big).
\end{align*}
Thus, we have
\begin{align*}
[Tu,Tv]&=T\big(\rho(Tu)v-\rho(d^{-1}v)d(Tu)\big)+T\big(\rho(d^{-1}u)d(Tv)-\rho(Tv)u\big)\\
&\quad-\Big(T\big(\rho(d^{-1}u)d(Tv)-\rho(d^{-1}v)d(Tu)\big)\Big)\\
&=T(\rho(Tu)v-\rho(Tv)u),
\end{align*}
which implies that $T$ is an $\huaO$-operator and thus $(T=N\circ d^{-1},d)$ is a {\rm KD}-structure.

Conversely, since $T$ is an $\huaO$-operator, we have
\begin{align*}
[Nx,Ny]=[T\circ dx,T\circ dy]=T\big(\rho(Nx)d(y)-\rho(Ny)d(x)\big).
\end{align*}
Since $d$ and $d\circ N$ are relative differential operators, we have
\begin{align*}
N[Nx,y]&=T\big(\rho(Nx)d(y)-\rho(y)d(Nx)\big),\\
N[x,Ny]&=T\big(\rho(x)d(Ny)-\rho(Ny)d(x)\big),\\
N^2[x,y]&=T\big(\rho(x)d(Ny)-\rho(y)d(Nx)\big).
\end{align*}
Thus, we have
\begin{align*}
&N\big([Nx,y]+[x,Ny]-N[x,y]\big)\\
&=T\big(\rho(Nx)d(y)-\rho(y)d(Nx)\big)+T\big(\rho(x)d(Ny)-\rho(Ny)d(x)\big)-\Big(T\big(\rho(x)d(Ny)-\rho(y)d(Nx)\big)\Big)\\
&=T\big(\rho(Nx)d(y)-\rho(Ny)d(x)\big)=[Nx,Ny],
\end{align*}
which implies that $N$ is a Nijenhuis operator and thus $(d,N=T\circ d)$ is a {\rm DN}-structure.
\end{proof}
\begin{thm}\label{col:hyper-DN}
Let $(d_1,d_2,d_3)$ be an $\varepsilon$-hyper relative differential operator on a Lie algebra $\g$. Then  $(d_i,N_k)$ are {\rm DN}-structures for all $i\neq k\in\mathbb{Z}_{3}$.
\end{thm}
\begin{proof}
By Prop. \ref{pro:hyper-KD}, $(T_i,d_k)$ are {\rm KD}-structures for all $i\neq k\in\mathbb{Z}_{3}$. By Prop. \ref{pro:hyper-KD-DN}, $(d_i,N_k)$ are {\rm DN}-structures for all $i\neq k\in\mathbb{Z}_{3}$.
\end{proof}

\subsection{Hyper relative differential operators, dual-Nijenhuis pairs and {\rm KN}-structures} 
Following \cite{HLS}), a pair $(N,S)$, where $N\in\gl(\g)$ and $S\in\gl(V)$, is called a \emph{dual-Nijenhuis pair} on a Lie algebra  $\g$ with a representation $(V;\rho)$ if $N$ is a Nijenhuis operator on $\g$ and $S$ satisfies the following condition:
\begin{equation}\label{eq:coNijpair}
\rho({Nx})(Sv)=S(\rho({Nx})v)+\rho({x})(S^2(v))-S(\rho({x})(Sv)).
\end{equation}
Let $(N,S)$ be a dual-Nijenhuis pair on a Lie algebra  $(\g,[\cdot,\cdot])$ with a representation $\rho$. Define ${\varrho}:\g\longrightarrow\gl(V)$ by
\begin{equation}
{\varrho}(x):=\rho(Nx)-[\rho(x),S],\quad x\in\g.
\end{equation}
It was shown in \cite{HLS} that  ${\varrho}$ is a representation of the Lie algebra $(\g,[\cdot,\cdot]_N)$ on $V$, where the Lie bracket $[\cdot,\cdot]_N$ is given by Eq. \eqref{eq:deformbracket}.

Let $T:V\to \g$ be an $\huaO$-operator on a Lie algebra $\g$ with respect to a representation $(V;\rho)$. Define a multiplication $\star^T$ on $V$ by
\begin{equation}
  u\star^T v:=\rho(Tu)(v),\quad \text{ for all }\, u,v\in V.
\end{equation}
Then, $(V,\star^T)$ is a pre-Lie algebra. We denote by $(V,[\cdot,\cdot]^T)$ the sub-adjacent Lie algebra of the pre-Lie algebra $(V,\star^T)$. More precisely,
\begin{equation}\label{eq:bracketT}
 ~[u,v]^T:=\rho(Tu)(v)-\rho(Tv)(u),\quad \text{ for all }\, u,v\in V.
\end{equation}
Moreover, $T$ is a Lie algebra homomorphism from $(V,[\cdot,\cdot]^T)$ to $(\g,[\cdot,\cdot])$, see \cite{Bai1} for more details.

Let $T:V\rightarrow\g$ be an $\huaO$-operator and $(N,S)$ be a dual-Nijenhuis pair on a Lie algebra  $\g$ with a representation $(V;\rho)$.  We define the bracket $[\cdot,\cdot]^T_S:\wedge^2V\longrightarrow V$ to be the deformed bracket of $[\cdot,\cdot]^T$ by $S$, i.e.
\begin{eqnarray*}
 {[u,v]}^T_S&:=&[S(u),v]^T+[u,S(v)]^T-S[u,v]^T,\quad \text{ for all } \, u,v\in V.\label{eq:defieq33}
\end{eqnarray*}
Define the bracket $\{\cdot,\cdot\}_{{\varrho}}^T:\wedge^2V\longrightarrow V$ by
\begin{eqnarray}
\label{eq:defieq23}\{u,v\}_{{\varrho}}^T&:=&{\varrho}(Tu)(v)-{\varrho}(Tv)(u),\quad \text{ for all }\, u,v\in V.
\end{eqnarray}

\begin{defi}{\rm (\cite{HLS})}
Let $T:V\rightarrow\g$ be an $\huaO$-operator and $(N,S)$ a dual-Nijenhuis pair on a Lie algebra  $\g$ with a representation $(V;\rho)$. The triple $(T,S,N)$ is called a \emph{Kupershmidt-dual-Nijenhuis structure ({\rm KN}-structure)} on $\g$ with respect to $(V;\rho)$ if $T$ and $(N,S)$ satisfy the following conditions
\begin{eqnarray}
\label{eq:TN1}N\circ T&=&T\circ S,\\
\label{eq:TN2} {[u,v]}^{N\circ T}&=&{[u,v]}^{T}_{S}.
\end{eqnarray}
\end{defi}
Note that if $(T,S,N)$ is a KN-structure, then $S$ is a Nijenhuis operator on the sub-adjacent Lie algebra $(V,[\cdot,\cdot]^T)$, and furthermore, the three brackets  ${[\cdot,\cdot]}^T_S,$ $\{\cdot,\cdot\}_{{\varrho}}^T$ and $[\cdot,\cdot]^{N\circ T }$ coincide and are all Lie brackets.

\begin{pro}{\rm (\cite{HLS})}\label{pro:TS1}
 Let $(T,S,N)$ be a {\rm KN}-structure with respect to the representation $(V;\rho)$. Then, we have
\begin{itemize}
\item[$\rm(i)$] $T$ is an $\huaO$-operator on the deformed Lie algebra $(\g,[\cdot,\cdot]_N)$ with respect to the representation $(V; {\varrho})$;
\item[$\rm(ii)$] $N\circ T$ is an $\huaO$-operator on the Lie algebra $\g$ with respect to the representation  $(V;\rho)$.
\end{itemize}
\end{pro}

\begin{defi}
Let $T_1,T_2: V\rightarrow \g$ be two $\huaO$-operators on a Lie algebra $\g$ with respect to a representation $(V;\rho)$. If for any $k_1,k_2\in \mathbb{Z}_+$, the operator $k_1T_1+k_2T_2$ is still an $\huaO$-operator, then $T_1$ and $T_2$ are called \emph{compatible}.
\end{defi}

The presence of a KN-structure gives rise to a hierarchy of compatible 
$\huaO$-operators.
\begin{pro}{\rm (\cite{HLS})}\label{pro:hierarchy}
Let $(T,S,N)$ be a {\rm KN}-structure with respect to the representation $(V;\rho)$. Then all $N^k\circ T$ are $\huaO$-operators with respect to the representation $(V;\rho)$, and for all $k,l\in\Nat$, $N^k\circ T$ and $N^l\circ T$ are compatible.
\end{pro}

\emptycomment{Conversely, compatible $\huaO$-operators  can give rise to {\rm KN}-structures.
\begin{pro}{\rm (\cite{HLS})}\label{pro:ComptoNS}
Let $T,T_1: V\longrightarrow \g$ be two $\huaO$-operators  on
a Lie algebra $\g$ with respect to a representation
$(V;\rho)$. Suppose that $T$ is invertible. If
$T$ and $T_1$ are compatible, then
\begin{itemize}
\item[$\rm(i)$]$(T,S=T^{-1}\circ T_1,N=T_1\circ T^{-1})$ is a {\rm KN}-structure;
\item[$\rm(ii)$]$(T_1,S=T^{-1}\circ T_1,N=T_1\circ T^{-1})$ is a {\rm KN}-structure.
\end{itemize}
\end{pro}}

\begin{pro}\label{pro:hyper-KN}
If $(T,d)$ is a {\rm KD}-structure, then $(T,S=d\circ T,N=T\circ d)$ is a {\rm KN}-structure.
\end{pro}
\begin{proof}
 Let $(T,d)$ be a {\rm KD}-structure.  It is obvious that $N\circ T=T\circ d\circ T=T\circ S$. First we show that $(N,S)$ is a dual-Nijenhuis pair. Since $T$ is an $\huaO$-operator and $d$ is a relative differential operator, we have 
 \begin{eqnarray*}
 dT\big(\rho(Tv) dx-\rho(Tdx) v\big)=d([Tv,Tdx])=\rho(Tv)dTdx-\rho(Tdx)dTv.
 \end{eqnarray*}
 Due to $S=d\circ T$, the above equation implies
 \begin{equation}\label{eq:hyper-KN1}
 S\rho(Tv) dx-S\rho(Tdx) v=\rho(Tv)Sdx-\rho(Tdx)Sv.
 \end{equation}
 Since $S\circ d=d\circ N$ is a relative differential operator, we have
$$S d[x,Tv]=\rho(x)(S dTv)-\rho(Tv)(S dx),$$
which implies that
\begin{align}\label{eq:kncollary2}
\rho(Tv)\big(Sdx\big)-S\big(\rho(Tv)(dx)\big)-\rho(x)\big(S^2(v)\big)+S\big(\rho(x)(Sv)\big)=0.
\end{align}
 Combining Eq. \eqref{eq:hyper-KN1} with Eq. \eqref{eq:kncollary2}, letting $T\circ d=N$, we have
\begin{align*}
\rho({Nx})(Sv)-S\big(\rho({Nx})v\big)=\rho({x})(S^2(v))-S\big(\rho({x})(Sv)\big),
\end{align*}
which means that $(N,S)$ is a dual-Nijenhuis pair.
 
 Since $T$ is an $\huaO$-operator and $d$ is a relative differential operator, we have
\begin{eqnarray*}
&&\rho(Tv)(Su)-\rho(Tu)(Sv)+S\big(\rho(Tu)v-\rho(Tv)u\big)\\
&=&\rho(Tv)(d Tu)-\rho(Tu)(d Tv)+d T\big(\rho(Tu)v-\rho(Tv)u\big)\\
&=&\rho(Tv)(dTu)-\rho(Tu)(dTv)+d[Tu,Tv]=0,
\end{eqnarray*}
which implies that
\begin{equation}\label{eq:KN1}
\rho(Tv)(Su)-\rho(Tu)(Sv)+S\big(\rho(Tu)v-\rho(Tv)u)\big)=0.
\end{equation}
By Eq. \eqref{eq:KN1}, we have
\begin{eqnarray*}
&&{[u,v]}^{N\circ T}-{[u,v]}^{T}_{S}={[u,v]}^{T\circ S}-{[u,v]}^{T}_{S}\\
&=&\rho\big(TS(u)\big)v-\rho\big(TS(v)\big)u-\Big(\rho\big(TS(u)\big)v-\rho(Tv)\big(S(u)\big)\\
&&+\rho(Tu)\big(S(v)\big)-\rho\big(T S(v)\big)u-S\big(\rho(Tu)v-\rho(Tv)u\big)\Big)\\
&=&\rho(Tv)\big(S(u)\big)-\rho(Tu)\big(S(v)\big)+S\big(\rho(Tu)v-\rho(Tv)u\big)\\
&=&0.
\end{eqnarray*}
Therefore, $(T,S=d\circ T,N=T\circ d)$ is a {\rm KN}-structure.
\end{proof}

\begin{thm}\label{thm:hyper-KN}
Let $(d_1,d_2,d_3)$ be an $\varepsilon$-hyper relative differential operator on a Lie algebra $\g$. Then, for all $i,k\in\mathbb{Z}_{3}$ such that $i\neq k$, $(T_i,S_k,N_k)$ are {\rm KN}-structures.
\end{thm}
\begin{proof}
Since $(d_1,d_2,d_3)$ is an $\varepsilon$-hyper relative differential operator, by Prop.  \ref{pro:hyper-KD}, $(T_{i},d_k)$ are {\rm KD}-structures for all $i\neq k\in\mathbb{Z}_{3}$. Furthermore, by Prop. \ref{pro:hyper-KN}, $(T_i,S_k,N_k)$ are {\rm KN}-structures for all $i\neq k\in\mathbb{Z}_{3}$.
\end{proof}

\begin{cor}
	Let $(d_1,d_2,d_3)$ be an $\varepsilon$-hyper relative differential operator on a Lie algebra $\g$. Then, for all $i\neq k\in\mathbb{Z}_{3}$, $N_k\circ T_i$ and 
	$ T_i$ are compatible $\huaO$-operators. In particular, $T_i$ and $T_{j}$ are compatible $\huaO$-operators for all $i,j\in\mathbb{Z}_{3} $.
\end{cor}	
\begin{proof}
	It follows by Theorem \ref{thm:hyper-KN} and Proposition \ref{pro:hierarchy}. 
	\end{proof}

\section{\texorpdfstring{$\varepsilon$-hyper relative differential operators with $\varepsilon_1\varepsilon_2\varepsilon_3=1$ and $\varepsilon_1\varepsilon_2\varepsilon_3=-1$}{epsilon-hyper relative differential operators with epsilon1, epsilon2, epsilon3=1 and epsilon1, epsilon2, epsilon3=-1}}\label{sec:2cases}
In this section, we study $\varepsilon$-hyper relative differential operators with $\varepsilon_1\varepsilon_2\varepsilon_3=1$ and $\varepsilon_1\varepsilon_2\varepsilon_3=-1$. More compatible structures  and equivalent descriptions for hyper relative differential operators are given.

\subsection{$\varepsilon$-hyper relative differential operators with $\varepsilon_1\varepsilon_2\varepsilon_3=1$}

 \begin{pro}
	Let $(d_1,d_2,d_3)$ be an $\varepsilon$-hyper relative differential operator on a Lie algebra $\g$ such that $\varepsilon_1\varepsilon_2\varepsilon_3=1$. Then, for all $i\in\mathbb{Z}_{3}$,  $(d_i, N_i)$ are {\rm DN}-structures.
\end{pro}
\begin{proof}
In order to prove that $(d_i, N_i)$ is a {\rm DN}-structure for any $i\in \mathbb{Z}_{3}$, we only need to prove that $d_i\circ N_i$ is a relative differential operator; that is, $d_i N_i [x,y]=\rho(x)(d_i N_i y)-\rho(y)(d_i N_i x)$ for all  $x,y\in \g$. Since $\varepsilon_1\varepsilon_2\varepsilon_3=1$, by Eq. \eqref{eq:hdnproperty1}, we have $d_i\circ N_i=S_i\circ d_i$. Letting $x=T_iu$ and $y=T_iv$, we have
\begin{eqnarray*}
  &&d_i N_i [x,y]-\rho(x)(d_i N_i y)+\rho(y)(d_i N_i x)\\
  &=&S_id_i[T_iu,T_iv]-\rho(T_iu)(S_i d_i T_iv)+\rho(T_iv)(S_id_i T_iu)\\
  &=&S_i\rho(T_iu )v-S_i\rho(T_iv)u-\rho(T_i u)(S_i v)+\rho(T_i v)(S_i u),
\end{eqnarray*}
which means that  $d_i\circ N_i$ is a relative differential operator if and only if the following equation holds:
\begin{equation}\label{eq:important relation}
S_i\rho(T_iu )v-S_i\rho (T_iv)u-\rho(T_i u)(S_i v)+\rho(T_i v)(S_i u)=0.
\end{equation}
By (vi) in Prop. \ref{pro:pro2.8}, we have $S_i=\varepsilon_id_{i-1}\circ T_{i+1}$. Considering that $T_i$ and $T_{i+1}$ are compatible and $d_{i-1}$ is a relative differential operator, and using  \cite[Proposition 4.2]{HLS}, we obtain 
\begin{eqnarray*}
[T_iu,T_{i+1}v]+[T_{i+1}u,T_iv]=T_i\big(\rho(T_{i+1}u)v-\rho(T_{i+1}v)u\big)+T_{i+1}\big(\rho(T_{i}u)v-\rho(T_{i}v)u\big),
\end{eqnarray*}
which implies
\begin{eqnarray*}
d_{i-1}[T_iu,T_{i+1}v]-d_{i-1}[T_iv, T_{i+1}u]=d_{i-1}T_i\big(\rho(T_{i+1}u)v-\rho(T_{i+1}v)u\big)+d_{i-1}T_{i+1}\big(\rho(T_{i}u)v-\rho(T_{i}v)u\big).
\end{eqnarray*}
Thus, 
\begin{eqnarray*}
&&S_i\rho(T_iu )v-S_i(\rho T_iv)u-\rho(T_i u)(S_i v)+\rho(T_i v)(S_i u)\\
&=&\varepsilon_i\Big( d_{i-1} T_{i+1} \rho(T_iu )v-d_{i-1} T_{i+1} \rho(T_iv )u-\rho(T_i u)(d_{i-1} T_{i+1} v)+\rho(T_i v)(d_{i-1} T_{i+1} u)  \Big)\\
&=&\varepsilon_i\Big( d_{i-1} T_{i+1} \rho(T_iu )v-d_{i-1} T_{i+1} \rho(T_iv )u-d_{i-1}[T_iu,T_{i+1} v]-\rho(T_{i+1} v)(d_{i-1} T_{i} u)\\
&&+d_{i-1}[T_iv,T_{i+1} u]+\rho(T_{i+1} u)(d_{i-1} T_{i} v) \Big)\\
&=& \varepsilon_i\Big(d_{i-1} T_{i+1} \rho(T_iu )v-d_{i-1} T_{i+1} \rho(T_iv )u-d_{i-1}T_i\big(\rho(T_{i+1}u)v-\rho(T_{i+1}v)u\big)\\
&&-d_{i-1}T_{i+1}\big(\rho(T_{i}u)v-\rho(T_{i}v)u\big)-\rho(T_{i+1} v)(d_{i-1} T_{i} u)+\rho(T_{i+1} u)(d_{i-1} T_{i} v)\Big)\\
&=&-\varepsilon_i\Big( d_{i-1} T_{i} \rho(T_{i+1}u )v-d_{i-1} T_{i} \rho(T_{i+1}v )u+\rho(T_{i+1} v)(d_{i-1} T_{i} u)-\rho(T_{i+1} u)(d_{i-1} T_{i} v) \Big)\\
&=&-\varepsilon_i\Big(S_{i+1}\rho(T_{i+1}u )v-S_{i+1}\rho(T_{i+1}v)u-\rho(T_{i+1} u)(S_{i+1} v)+\rho(T_{i+1} v)(S_{i+1} u)  \Big)\\
&=&\varepsilon_i\varepsilon_{i+1}\Big(S_{i+2}\rho(T_{i+2}u )v-S_{i+2}\rho(T_{i+2}v)u-\rho(T_{i+2} u)(S_{i+2} v)+\rho(T_{i+2} v)(S_{i+2} u)  \Big)\\
&=&-\varepsilon_i\varepsilon_{i+1}\varepsilon_{i+2}\Big(S_{i}\rho(T_{i}u )v-S_{i}\rho(T_{i}v)u-\rho(T_{i} u)(S_{i} v)+\rho(T_{i} v)(S_{i} u)  \Big).
\end{eqnarray*}
Due to $\varepsilon_1\varepsilon_2\varepsilon_3=1$, we obtain $S_i\rho(T_iu )v-S_i\rho(T_iv)u-\rho(T_i u)(S_i v)+\rho(T_i v)(S_i u)=0.$ Therefore, for any $i\in \mathbb{Z}_{3}$, $(d_i, N_i)$ is a {\rm DN}-structure. 
\end{proof}
\begin{cor}
Let $(d_1,d_2,d_3)$ be an $\varepsilon$-hyper relative differential operator on a Lie algebra $\g$ such that $\varepsilon_1\varepsilon_2\varepsilon_3=1$. Then, for all $i\in\mathbb{Z}_{3}$, $(K_i=N_i\circ T_i,d_i)$ is a {\rm KD}-structure and
  $(T_i,S_i,N_i)$ are {\rm KN}-structures.
\end{cor}
\begin{proof}
Since $(d_i, N_i)$ is a {\rm DN}-structure, for any $i\in \mathbb{Z}_{3}$, by Prop. \ref{pro:hyper-KD-DN},   $(K_i=N_i\circ T_i,d_i)$ is a {\rm KD}-structure. Furthermore, by Prop. \ref{pro:hyper-KN}, $(T_i,S_i,N_i)$ is a {\rm KN}-structure.
\end{proof}

\begin{cor}
	Let $(d_1,d_2,d_3)$ be an $\varepsilon$-hyper relative differential operator on a Lie algebra $\g$ such that $\varepsilon_1\varepsilon_2\varepsilon_3=1$. Then, $h^\flat:\g\to V$ is a relative differential operator. 
\end{cor}
\begin{proof}
Since $(K_3=N_3\circ T_3,d_3)$ is a {\rm KD}-structure, $K_3=N_3\circ T_3=T_2\circ d_1\circ T_3=\varepsilon_2\varepsilon_3 (h^\flat)^{-1}$ is an $\huaO$-operator and thus $h^\flat:\g\to V$ is a relative differential operator.
\end{proof}

\begin{pro}
	Let $(d_1,d_2,d_3)$ be an $\varepsilon$-hyper relative differential operator on a Lie algebra $\g$ such that $\varepsilon_1\varepsilon_2\varepsilon_3=1$. Then, for all $i\in\mathbb{Z}_{3}$,  $(h^\flat,T_i\circ h^\flat=\varepsilon_{i-1}N_i)$ and $(d_i, (h^\flat)^{-1}\circ d_i=\varepsilon_{i+1}N_i)$ are {\rm DN}-structures. In particular, $(h^\flat,N_i)$ are also {\rm DN}-structures.
\end{pro}
\begin{proof}
By Eq. \eqref{eq:hdnproperty2}, we have $T_i\circ h^\flat=\varepsilon_{i-1}N_i$ and thus $T_i\circ h^\flat$ is a Nijenhuis operator on the Lie algebra $\g$. Furthermore, by Eq. \eqref{eq:hdnproperty3}, we have $h^\flat\circ T_i\circ h^\flat=\varepsilon_{i-1}h^\flat\circ N_i=\varepsilon_i d_i$ and thus $h^\flat\circ T_i\circ h^\flat$ is a relative differential operator on the Lie algebra $\g$. Therefore, for any $i\in\mathbb{Z}_{3}$,  $(h^\flat,T_i\circ h^\flat)$ is a {\rm DN}-structure.

By Eq. \eqref{eq:hdnproperty3}, we have $(h^\flat)^{-1}\circ d_i=\varepsilon_{i+1}N_i$ and thus $(h^\flat)^{-1}\circ d_i$ is a Nijenhuis operator on the Lie algebra $\g$. Furthermore, we have $d_i\circ (h^\flat)^{-1}\circ d_i=\varepsilon_{i+1}d_i\circ N_i$. Since $(d_i, N_i)$ is a {\rm DN}-structure, $d_i\circ N_i$ is a relative differential operator and thus $d_i\circ (h^\flat)^{-1}\circ d_i$ is a relative differential operator. Therefore, for any $i\in\mathbb{Z}_{3}$,  $(h^\flat,T_i\circ h^\flat)$ is a {\rm DN}-structure.
\end{proof}
By Prop. \ref{pro:hyper-KD-DN} and Prop. \ref{pro:hyper-KN}, we have 
\begin{cor}
Let $(d_1,d_2,d_3)$ be an $\varepsilon$-hyper relative differential operator on a Lie algebra $\g$ such that $\varepsilon_1\varepsilon_2\varepsilon_3=1$. Then, for all $i\in\mathbb{Z}_{3}$,  $(T_i,h^\flat)$ and $((h^\flat)^{-1},d_i)$ are {\rm KD}-structures. Furthermore, $(T_i,h^\flat\circ T_i=\varepsilon_{i-1}S_i,T_i \circ h^\flat=\varepsilon_{i-1}N_i)$ and $((h^\flat)^{-1},d_i\circ (h^\flat)^{-1}=\varepsilon_{i+1}S_i,(h^\flat)^{-1}\circ d_i=\varepsilon_{i+1}N_i)$ are {\rm KN}-structures. In particular, $((h^\flat)^{-1},S_i,N_i)$ are also {\rm KN}-structures. 
\end{cor}

\begin{cor}
	Let $(d_1,d_2,d_3)$ be an $\varepsilon$-hyper relative differential operator on a Lie algebra $\g$ such that $\varepsilon_1\varepsilon_2\varepsilon_3=1$. Then, for all $i\in\mathbb{Z}_{3}$,  $T_i$  and  $(h^\flat)^{-1}$ are compatible.
\end{cor}
\begin{proof}
Since for any $i\in\mathbb{Z}_{3}$, $((h^\flat)^{-1},S_i,N_i)$ is a {\rm KN}-structure,  by Prop. \ref{pro:hierarchy}, we have $(h^\flat)^{-1}$ and $N_i\circ (h^\flat)^{-1}$ are compatible. Also, by Eq.\eqref{eq:hdnproperty1}, we have $N_i\circ (h^\flat)^{-1}=\varepsilon_i\varepsilon_{i+1} T_i$. Thus $(h^\flat)^{-1}$ and $T_i$ are compatible. 
	\end{proof}

\subsection{$\varepsilon$-hyper relative differential operators with $\varepsilon_1\varepsilon_2\varepsilon_3=-1$}

In this section, we discuss two kinds of  $\varepsilon$-hyper relative differential operators with $\varepsilon_1\varepsilon_2\varepsilon_3=-1$. 
\begin{defi}
Let $(d_1,d_2,d_3)$ be an $\varepsilon$-hyper relative differential operator on a Lie algebra $\g$ such that $\varepsilon_1\varepsilon_2\varepsilon_3=-1$. If $\varepsilon_1=\varepsilon_2=\varepsilon_3=-1$, then  $(d_1,d_2,d_3)$ is called a \emph{hyper relative differential operator} on $\g$, while otherwise, $(d_1,d_2,d_3)$ is called a \emph{para-hyper relative differential operator} on $\g$.
\end{defi}

With a cyclic permutation of the indices, all para-hyper relative differential operators satisfy $\varepsilon_1=\varepsilon_2=1$ and $\varepsilon_3=-1$. Thus,  we just need to study para-hyper relative differential operators of this form in the following sections. 

\begin{thm}\label{thm:equivalence}
		Let $\g$ be a Lie algebra and $V$ be a vector space.   Then  $(d_1,d_2,d_3)$ is a (para)-hyper relative differential operator if and only if there exist an invertible linear map $h^\flat:\g\to V$ and (para-)complex structures $I_i:\g\to\g$  for $i=1,2$ satisfying $I_1\circ I_2=-I_2\circ I_1,$ and $d_i=h^\flat\circ I_i$ are relative differential operators  for all $i=1,2,3$ with $I_3=I_1\circ I_2$. 
	\end{thm}
\begin{proof}
Let	$(d_1,d_2,d_3)$ be a hyper relative differential operator. Letting $h^\flat=d_{3}\circ T_{1}\circ d_{2}$, $I_1=N_1$ and $I_2=N_2$, by  (iii) in Prop.  \ref{pro:pro2.8}, we have $I_1\circ I_2=-I_2\circ I_1$ and $I_3=N_3$.  By Eq.\eqref{eq:hdnproperty3}, we have $d_i=h^\flat\circ N_i=h^\flat\circ I_i$ and thus $h^\flat\circ I_i $ are relative differential operators  for all $i=1,2,3$. 

Conversely, by $I_3=I_1\circ I_2$, we have 
$$d_3=h^\flat\circ I_3=h^\flat\circ I_1\circ I_2=d_1\circ I_2,$$
which implies that $I_2=N_2=T_1\circ d_3$. Similarly, we  have $I_1=N_1$ and $I_2=N_2$. Since $I_1^2=I_2^2=-\Id_\g$, we have $I_3^2=-\Id_\g$.  Hence $(d_1,d_2,d_3)$ is a hyper relative differential operator.

The second claim can be proved similarly.
\end{proof}	

\emptycomment{
{\bf Recall a complex product on a Lie algebra $\g$ is a pair $(E,J)$ such that $E$ is a product structure and $J$ is a complex structure on the Lie algebra $\g$ satisfying $JE=-EJ$. A complex product structure $(E,J)$ on $\g$ gives rise to a double Lie algebra $(\g,\g_+,\g_-)$ such that $E|_{\g_+}=\id_{\g_+},~E|_{\g_-}=-\id_{\g_-}$ and $J\g_+=\g_-$. Conversely, one has
\begin{pro}{\rm(\cite{Andrada})}
A double Lie algebra $(\g,\g_+,\g_-)$ with {\rm dim} $\g_+$={\rm dim} $\g_-$ induces a complex product structure on the Lie algebra $\g$, where the complex product structure $(E,J)$ is defined by
$$E|_{\g_+}=\id_{\g_+},~E|_{\g_-}=-\id_{\g_-},~J(x+a)=-\varphi^{-1}(a)+\varphi(x),\quad x\in\g_+,a\in\g_-,$$
if and only if there exists a linear isomorphism $\varphi:\g_+\rightarrow\g_-$ such that 
\begin{eqnarray}
    \varphi[x,y]+\varphi^{-1}[\varphi(x),\varphi(y)]=[\varphi(x),y]+[x,\varphi(y)],\quad x,y\in\g_+.
\end{eqnarray}
\end{pro}

If $(d_1,d_2,d_3)$ is a para-hyper relative differential operator, then $(E=I_1,J=I_3)$ is a complex structure on $\g$. Then $(\g=\g_+\oplus\g_-,\g_+,\g_-)$ is a double Lie algebra and $V=V_+\oplus V_-$ such that $h^\flat|_{\g_+}=V_+$ and $h^\flat|_{\g_1}=V_1$. Then $d_1$ is a relative differential differential operator if and only if ???, $d_2$ is a relative differential differential operator if and only if ???, and $d_3$ is a relative differential differential operator if and only if ???
}}

\section{$\varepsilon$-hyper symplectic structures and $\varepsilon$-hyper Hessian structures}\label{sec:hyper2}

\subsection{$\varepsilon$-hyper symplectic structures}
Let $\omega$ be a symplectic structure on a Lie algebra $\g$. Recall that $\omega^\natural:\g\rightarrow\g^*$ defined by Eq. \eqref{eq:sym-diff} is a relative differential operator with respect to the coadjoint representation $(\g^*;\ad^*)$.  

\begin{defi}
	Let $\omega_1,\omega_2$ and $\omega_3$ be symplectic structures on the Lie algebra $\g$. If the triple
	$(\omega^\natural_1,\omega^\natural_2,\omega^\natural_3)$ is an  $\varepsilon$-hyper relative differential operator with respect to the coadjoint representation $(\g^*;\ad^*)$, then $(\omega_1,\omega_2,\omega_3)$ is called an  \emph{$\varepsilon$-hyper  symplectic structure}. 
\end{defi}

Recall that a \emph{Hermitian structure} on a Lie algebra $\g$ is a pair $(h,I)$, where $h$ is a pseudo-metric and $I:\g\to \g$ is a complex structure satisfying 
\begin{equation}\label{eq:sym-invariant1}
h(I(x),I(y))=h(x,y),\quad \text{ for all }\, x,y\in \g,
\end{equation}
while a \emph{para-Hermitian structure} on a Lie algebra $\g$ is a pair $(h,I)$, where $h$ is a pseudo-metric and $I:\g\to \g$ is a para-complex structure satisfying 
\begin{equation}\label{eq:sym-invariant2}
h(I(x),I(y))=-h(x,y),\quad \text{ for all }\, x,y\in \g.
\end{equation}

\begin{defi}{\rm (\cite{Hit})}
	Let $(h,I_1)$ and $(h,I_2)$ be (para-)Hermitian structures on a Lie algebra $\g$ and $I_3=I_1I_2$. If $I_1I_2=-I_2I_1$ and $\omega_i$ defined by 
	\begin{equation}\label{eq:her-sym}
		\omega_i(x,y)=h(I_i(x),y),\quad \text{ for all }\, x,y\in \g,
	\end{equation}
	are symplectic structures on $\g$ for  $i=1,2,3$, then $(h,I_1,I_2,I_3)$ is called a \emph{(para-)hyper K\"ahler structure} on the Lie algebra $\g$.
\end{defi}	
\begin{rmk}
The notion of hyper symplectic structure on a Lie algebra introduced in \cite{Andrada05} corresponds to the para-hyper K\"ahler structure appearing here.
\end{rmk}

\begin{pro}\label{pro:hyper-kahler}
	\begin{itemize}
	\item[\rm (i)]The quadruple $(h,I_1,I_2,I_3)$ is a hyper K\"ahler structure on a Lie algebra $\g$ if and only  if $(\omega_1,\omega_2,\omega_3)$ is an  $\varepsilon$-hyper  symplectic structure with $\varepsilon=(-1,-1,-1)$, where $\omega_i$, for  $i=1,2,3$, are given by \textup{Eq}. \eqref{eq:her-sym}.
	\item[\rm (ii)]The quadruple $(h,I_1,I_2,I_3)$ is a para-hyper K\"ahler structure on the Lie algebra $\g$ if and only  if $(\omega_1,\omega_2,\omega_3)$ is an  $\varepsilon$-hyper  symplectic structure with $\varepsilon=(1,1,-1)$, where $\omega_i$, for  $i=1,2,3$, are given by \textup{Eq.}  \eqref{eq:her-sym}.
	\end{itemize}
\end{pro}	
\begin{proof}
Let	$(h,I_1,I_2,I_3)$ be a hyper K\"ahler structure on the Lie algebra $\g$. By Theorem \ref{thm:equivalence}, $(\omega^\natural_1,\omega^\natural_2,\omega^\natural_3)$ is an  $\varepsilon$-hyper relative differential operator with respect to the coadjoint representation $(\g^*;\ad^*)$ with $\varepsilon=(-1,-1,-1)$. Furthermore, for $i=1,2$, by Eq. \eqref{eq:sym-invariant1} and $I_i^2=-\id_\g$, we have
\begin{align*}
\omega_i(x,y)=h(I_i(x),y)=-h(x,I_i(y))=-h(I_i(y),x)=-\omega_i(y,x),
\end{align*}
and by $I_1I_2=-I_2I_1$, we also have 
\begin{align*}
\omega_3(x,y)=h(I_3(x),y)=h(I_1I_2(x),y)=h(x,I_2I_1(y))=-h(I_3(y),x)=-\omega_3(y,x),
\end{align*}
which implies that $\omega_1$, $\omega_2$ and $\omega_3$ are skew-symmetric and hence are symplectic structures on the Lie algebra $\g$. Note that, for $i=1,2,3$, the cocycle condition satisfied by $\omega_i$ is equivalent to $d_i$ being a relative differential operator  with respect to the coadjoint representation $(\g^*;\ad^*)$.   Thus,  $(\omega_1,\omega_2,\omega_3)$ is an $\varepsilon$-hyper  symplectic structure with $\varepsilon=(-1,-1,-1)$.  

Conversely, if $(\omega_1,\omega_2,\omega_3)$ is an  $\varepsilon$-hyper  symplectic structure with $\varepsilon=(-1,-1,-1)$, by Theorem \ref{thm:equivalence}, there exist an invertible linear map $h^\flat:\g\to \g^*$, and complex structures $I_i:\g\to\g$  for $i=1,2$ such that $I_1\circ I_2=-I_2\circ I_1,$ and $\omega^\natural_i=h^\flat\circ I_i$. Set $h(x,y)=\langle h^\flat (x),y\rangle$ for $x,y\in \g$. By the fact that  $h^\flat=\omega^\natural_3\circ (\omega^\natural_1)^{-1}\circ \omega^\natural_2=-\omega^\natural_2\circ (\omega^\natural_1)^{-1}\circ \omega^\natural_3$ and $\omega_i$ are skew-symmetric, we have
$$h(x,y)=\langle h^\flat (x),y\rangle=\langle \omega^\natural_3\circ (\omega^\natural_1)^{-1}\circ \omega^\natural_2(x),y\rangle =-\langle x, \omega^\natural_2\circ (\omega^\natural_1)^{-1}\circ \omega^\natural_3(y) \rangle =h(y,x),$$
and 
 for $ i=1,2$, we have
$$h(I_i(x),I_i(y))=\omega_i(x,I_i(y))=-\omega_i(I_i(y),x)=h(x,y).$$
Thus, $(h,I_1)$ and $(h,I_2)$ are  Hermitian structures on the Lie algebra $\g$. 

The second claim can be proved similarly. 
\end{proof}

\begin{defi}
	Let $d_1,d_2$ and $d_3$ be invertible derivations on the Lie algebra $\g$. If the triple
		$(d_1,d_2,d_3)$ is an  $\varepsilon$-hyper relative differential operator with respect to the adjoint representation $(\g;\ad)$, then $(d_1,d_2,d_3)$ is called an  \emph{$\varepsilon$-hyper  differential operator} on the Lie algebra $\g$.
	\end{defi}

Let $\g$ be a Lie algebra with a non-degenerate, $\ad$-invariant, symmetric bilinear form $B\in\g\otimes\g$. Then $B$ induces a bijective linear map $B^\sharp:\g\to\g^*$ given by
\begin{equation}
	\langle B^\sharp(x),y\rangle=B(x,y),\quad \text{ for all }\, x,y\in \g.
\end{equation}
By the $\ad$-invariance of $B$, we have
\begin{equation}\label{eq:adinv}
	\ad^*_x B^\sharp (y)=B^\sharp(\ad_xy).
\end{equation}
A \emph{skew-symmetric endomorphism of $(\g,B)$} is a linear map $\varphi:\g\to \g$ such that $ B^\sharp\circ \varphi:\g^*\longrightarrow\g$ is skew-symmetric, that is, 
$$\langle B^\sharp\circ \varphi(x),y\rangle =-\langle B^\sharp\circ \varphi(y),x\rangle ,\quad \text{ for all } \, x,y\in\g.$$
\begin{thm}\label{thm:rmatrix-RBN}
	Let $d_i$ be skew-symmetric endomorphisms of $(\g,B)$ and set $\omega_i(x,y)=B(d_i(x),y)$ for all $i=1,2,3$.  Then $(d_1,d_2,d_3)$ is an  $\varepsilon$-hyper  differential operator if and only if $(\omega_1,\omega_2,\omega_3)$ is  an  $\varepsilon$-hyper  symplectic structure. 
\end{thm}
\begin{proof}
Let $(d_1,d_2,d_3)$ be an  $\varepsilon$-hyper  differential operator. Since $d_i$, for all $i=1,2,3$, are skew-symmetric endomorphisms and invertible, then $\omega_i$, for all $i=1,2,3$, are skew-symmetric and nondegenerate. Since $d_i$ are derivations on $\g$, by $\ad$-invariance of $B$,  we have
\begin{eqnarray*}
&&	\omega_i([x,y],z)+	\omega_i([z,x],y)+	\omega_i([y,z],x)\\
&=&B(d_i[x,y],z)-B(d_i(y),[z,x])-B(d_i(x),[y,z])\\
&=&B(d_i[x,y],z)-\langle \ad_x^* B^\sharp(d_i(y)),z\rangle +\langle \ad_y^* B^\sharp(d_i(x)),z\rangle \\
&=&B(d_i[x,y]-[d_i(x),y]-[x,d_i(y)],z)=0,
\end{eqnarray*}
which implies that $\omega_i$ are symplectic structures on the Lie algebra $\g$.  Also, we have 
$N_i=(\omega_{i-1}^\natural)^{-1}\circ \omega_{i+1}^\natural=d_{i-1}^{-1}\circ d_{i+1}.$
Thus $(\omega_1,\omega_2,\omega_3)$ is  an  $\varepsilon$-hyper  symplectic structure. 

The converse can be proved similarly. 
\end{proof}

By Theorem \ref{thm:rmatrix-RBN} and Proposition \ref{pro:hyper-kahler}, we have 
\begin{cor}
		Let $d_i$ be skew-symmetric endomorphisms of $(\g,B)$ for all $i=1,2,3$. Then 
			\begin{itemize}
			\item[\rm (i)]$(d_1,d_2,d_3)$ is an  $\varepsilon$-hyper  differential operator with $\varepsilon=(-1,-1,-1)$ if and only if the quadruple $(h,I_1,I_2,I_3)$ is a hyper K\"ahler structure on the Lie algebra $\g$, where $h^\flat=d_3\circ d_1^{-1}\circ d_2$ and $I_i=(d_{i-1})^{-1}\circ d_{i+1}$ for $i\in \mathbb{Z}_3$.
			\item[\rm (ii)]$(d_1,d_2,d_3)$ is an  $\varepsilon$-hyper  differential operator with  $\varepsilon=(1,1,-1)$ if and only if the quadruple $(h,I_1,I_2,I_3)$ is a para-hyper K\"ahler structure on the Lie algebra $\g$, where $h^\flat=-d_3\circ d_1^{-1}\circ d_2$ and $I_i=(d_{i-1})^{-1}\circ d_{i+1}$ for $i\in \mathbb{Z}_3$.
		\end{itemize}
\end{cor}	
\begin{proof}
(i) By Theorem \ref{thm:rmatrix-RBN}, $(d_1,d_2,d_3)$ is an  $\varepsilon$-hyper  differential operator with $\varepsilon=(-1,-1,-1)$ if and only if $(\omega_1,\omega_2,\omega_3)$ is  an  $\varepsilon$-hyper  symplectic structure with $\omega_i(x,y)=B(d_i(x),y)$ for all $i=1,2,3$. Furthermore, by Proposition \ref{pro:hyper-kahler}, $(\omega_1,\omega_2,\omega_3)$ is  an  $\varepsilon$-hyper  symplectic structure with $\varepsilon=(-1,-1,-1)$ and  $\omega_i(x,y)=B(d_i(x),y)$ for all $i=1,2,3$ if and only if the quadruple $(h,I_1,I_2,I_3)$ is a hyper K\"ahler structure on a Lie algebra $\g$, where $h(x,y)=d_3\circ d_1^{-1}\circ d_2$ and $I_i=(d_{i-1})^{-1}\circ d_{i+1}$ for $i\in \mathbb{Z}_3$.

Part (ii) can be proved similarly. 
\end{proof}

\subsection{$\varepsilon$-hyper Hessian structures}
Let $\mathscr{B}$ be a Hessian structure on the pre-Lie algebra $\g$. Recall that $\mathscr{B}^\natural:\g\rightarrow\g^*$ defined by \eqref{eq:hess-diff} is a relative differential operator with respect to the coregular representation $(\g^*,L^*)$.  
\begin{defi}
	Let $\mathscr{B}_1,\mathscr{B}_2$ and $\mathscr{B}_3$ be Hessian structures on the pre-Lie algebra $\g$. If the triple
	$(\mathscr{ B}^\natural_1,\mathscr{B}^\natural_2,\mathscr{ B}^\natural_3)$ is an  $\varepsilon$-hyper relative differential operator on the Lie algebra $\g^c$ with respect to the coregular representation $(\g^*;L^*)$, then $(\mathscr{B}_1,\mathscr{B}_2,\mathscr{B}_3)$ is called an  \emph{ $\varepsilon$-hyper  Hessian structure}. 
\end{defi}
\begin{ex}{\rm Consider the 4-dimensional pre-Lie algebra given in the basis $\{e_1, e_2, e_3, e_4\}$ by
\[
e_3 \cdot e_1= e_2, \quad e_3 \cdot e_2=- e_1.
\]
Let us define:
\[
\begin{array}{lcl}
\mathscr{B}_1 & = & e_1^* \otimes e_1^*+e_2^* \otimes e_2^* + 2 e_3^*\otimes e_3^*+e_3^*\otimes e_4^*+e_4^*\otimes e_3^*,\\
\mathscr{B}_2 & = & -e_1^* \otimes e_1^*-e_2^* \otimes e_2^* + 2 e_3^*\otimes e_3^*+e_4^*\otimes e_4^*+e_3^*\otimes e_4^*+e_4^*\otimes e_3^*,\\
\mathscr{B}_3 & = & i \, e_1^* \otimes e_1^*+i\, e_2^* \otimes e_2^* + 2 i\, e_3^*\otimes e_3^*+i\,  e_4^*\otimes e_4^*+i\, e_3^*\otimes e_4^*+i\, e_4^*\otimes e_3^*.\\
\end{array}
\]
A direct computation shows that $(\mathscr{B}_1,\mathscr{B}_2,\mathscr{B}_3)$ is an $\varepsilon$-hyper  Hessian structure for $\varepsilon=(-1,-1,1).$}
\end{ex}

\begin{defi} 
	An \emph{anti-Hermitian structure} on a Lie algebra $\g$ is a pair $(\omega,I)$, where $\omega$ is a non-degenerate skew-symmetric bilinear and $I:\g\to \g$ is a complex structure satisfying 
\begin{equation}\label{eq:skew-invariant1}
\omega(I(x),I(y))=\omega(x,y),\quad \text{ for all }\, x,y\in \g,
\end{equation}
while a \emph{para-anti-Hermitian structure} on a Lie algebra $\g$ is a pair $(\omega,I)$, where $\omega$ is a non-degenerate skew-symmetric bilinear and $I:\g\to \g$ is a para-complex structure satisfying 
\begin{equation}\label{eq:skew-invariant2}
\omega(I(x),I(y))=-\omega(x,y),\quad \text{ for all } \, x,y\in \g.
\end{equation}
\end{defi}
\begin{ex}{\rm Consider the 4-dimensional symplectic Lie algebra $L$ given in the basis $\{e_1, e_2, e_3, e_4\}$ by 
\[
[e_1, e_2]=e_3, \quad \omega=e_2^*\wedge e_3^*- e_1^*\wedge e_4^*.
\]
There exists a complex structure $I$ on $L$ given by 
\[
I(e_1)=-e_2, \quad I(e_2)=e_1, \quad I(e_3)=-e_4, \quad I(e_4)=e_3.
\]
Moreover, a direct computation shows that Eq. \eqref{eq:skew-invariant1} hold.
}
\end{ex}
\begin{defi}
Let $\g$ be a pre-Lie algebra.  Let $(\omega,I_1)$ and $(\omega,I_2)$ be (para-)anti-Hermitian structures on the sub-adjacent Lie algebra $\g^c$ and $I_3=I_1I_2$. If $I_1I_2=-I_2I_1$ and $\mathscr{B}_i\in \wedge^2\g$ defined by 
	\begin{equation}\label{eq:her-Hes}
		\mathscr{B}_i(x,y)=\omega(I_i(x),y),\quad \text{ for all } \, x,y\in \g,
	\end{equation}
	are Hessian structures on the pre-Lie algebra $\g$ for  $i=1,2,3$, then $(\omega,I_1,I_2,I_3)$ is called a \emph{ (para-)hyper anti-K\"ahler structure} on the pre-Lie algebra $\g$.
\end{defi}	
\begin{pro}\label{pro:hyper-hes}
	\begin{itemize}
	\item[\rm (i)]The quadruple $(\omega,I_1,I_2,I_3)$ is a hyper anti-K\"ahler structure on the pre-Lie algebra $\g$ if and only  if $(\mathscr{B}_1,\mathscr{B}_2,\mathscr{B}_3)$ is an  $\varepsilon$-hyper  Hessian structure with $\varepsilon=(-1,-1,-1)$, where $\mathscr{B}_i$ are given by \eqref{eq:her-Hes} for  $i=1,2,3$.
	\item[\rm (ii)]The quadruple $(\omega,I_1,I_2,I_3)$ is a para-hyper anti-K\"ahler structure on the pre-Lie algebra $\g$ if and only  if $(\mathscr{B}_1,\mathscr{B}_2,\mathscr{B}_3)$ is an  $\varepsilon$-hyper  Hessian structure with $\varepsilon=(1,1,-1)$, where $\mathscr{B}_i$ are given by \eqref{eq:her-Hes} for  $i=1,2,3$.
	\end{itemize}
\end{pro}	
\begin{proof}
Let	$(\omega,I_1,I_2,I_3)$ be a hyper Hessian structure on the pre-Lie algebra $\g$. By Theorem \ref{thm:equivalence}, $(\mathscr{ B}^\natural_1,\mathscr{ B}^\natural_2,\mathscr{ B}^\natural_3)$ is an  $\varepsilon$-hyper relative differential operator with respect to the coregular representation $(\g^*;L^*)$ with $\varepsilon=(-1,-1,-1)$, where $h^\flat=\omega^\natural$. Furthermore, for $i=1,2$, by Eq. \eqref{eq:skew-invariant1} and $I_i^2=-\id_\g$, we have
\begin{align*}
\mathscr{ B}_i(x,y)=\omega(I_i(x),y)=-\omega(x,I_i(y))=\omega(I_i(y),x)=\mathscr{B}_i(y,x),
\end{align*}
and by $I_1I_2=-I_2I_1$, we also have
\begin{align*}
\mathscr{ B}_3(x,y)=\omega(I_3(x),y)=\omega(I_1I_2(x),y)=\omega(x,I_2I_1(y))=\omega(I_3(y),x)=\mathscr{B}_3(y,x),
\end{align*}
which implies that $\mathscr{B}_1$, $\mathscr{B}_2$ and $\mathscr{B}_3$ are symmetric  and hence are Hessian structures on the pre-Lie algebra $\g$. Thus, $(\mathscr{B}_1,\mathscr{B}_2,\mathscr{B}_3)$ is an  $\varepsilon$-hyper Hessian structure with $\varepsilon=(-1,-1,-1)$.  

Conversely, if $(\mathscr{ B}_1,\mathscr{ B}_2,\mathscr{B}_3)$ is an  $\varepsilon$-hyper  Hessian structure with $\varepsilon=(-1,-1,-1)$, by Theorem \ref{thm:equivalence}, there exist an invertible linear map $h^\flat:\g\to V$, and complex structures $I_i:\g\to\g$  for $i=1,2$ such that $I_1\circ I_2=-I_2\circ I_1,$ and $\mathscr{ B}^\natural_i=h^\flat\circ I_i$. Set $\omega(x,y)=\langle h^\flat (x),y\rangle$ for $x,y\in \g$. Note that $h^\flat=\mathscr{ B}^\natural_3\circ (\mathscr{ B}^\natural_1)^{-1}\circ \mathscr{B}^\natural_2=-\mathscr{ B}^\natural_2\circ (\mathscr{ B}^\natural_1)^{-1}\circ \mathscr{B}^\natural_3$. By the symmetry of $\mathscr{B}_i$, we have
$$\omega(x,y)=\langle h^\flat (x),y\rangle=\langle \mathscr{ B}^\natural_3\circ (\mathscr{ B}^\natural_1)^{-1}\circ \mathscr{ B}^\natural_2(x),y\rangle =\langle x, \mathscr{ B}^\natural_2\circ (\mathscr{ B}^\natural_1)^{-1}\circ \mathscr{B}^\natural_1(y) \rangle =-\omega(y,x),$$
and 
 for $ i=1,2$, we have 
$$\omega(I_i(x),I_i(y))=\mathscr{ B}_i(x,I_i(y))=\mathscr{ B}_i(I_i(y),x)=-\omega(y,x)=\omega(x,y).$$
Thus $(\omega,I_1)$ and $(\omega,I_2)$ are  anti-Hermitian structures on the Lie algebra $\g^c$. 

The second claim can be proved similarly. 
\end{proof}	

Recall that a \emph{ derivation} on a pre-Lie algebra $\g$ is a linear map $d:\g\rightarrow \g$ satisfying $d(x\cdot y)=d(x)\cdot y+x\cdot d(y)$ for all $x,y\in \g$. 

\begin{defi}
Let $d_1,d_2$ and $d_3$ be invertible derivations on the pre-Lie algebra $\g$. If the triple
		$(d_1,d_2,d_3)$ is an  $\varepsilon$-hyper relative differential operator with respect to the regular representation $(\g;L)$, then $(d_1,d_2,d_3)$ is called an  \emph{ $\varepsilon$-hyper  differential operator} on the pre-Lie algebra $\g$.
	\end{defi}

A nondegenerate skew-symmetric bilinear form $\omega\in \g^*\otimes \g^*$ on a pre-Lie algebra $\g$ is called \emph{ invariant} if 
\begin{equation}
\omega(x\cdot y,z)+\omega(y,[x,z])=0,\quad \text{ for all } \, x,y,z\in \g.
\end{equation}

Let $\g$ be a pre-Lie algebra with a non-degenerate, invariant, skew-symmetric bilinear form $\omega\in \g^*\otimes \g^*$. Then $\omega$ induces a bijective linear map $\omega^\natural:\g\to\g^*$ given by $\langle \omega^\natural(x),y\rangle=\omega(x,y)$ for $x,y\in \g.$

By the invariance of $\omega$, we have
\begin{equation}\label{eq:Linv}
	\ad^*_x \omega^\natural (y)=\omega^\natural(L_xy),~\mbox{or equivalently,}~ L^*_x \omega^\natural (y)=\omega^\natural(\ad_xy).
\end{equation}
A \emph{symmetric endomorphism of $(\g,\omega)$} is a linear map $\omega:\g\to \g$ such that $ \omega^\natural\circ \varphi:\g^*\rightarrow\g$ is symmetric, that is, 
$$\langle \omega^\natural \varphi(x),y\rangle =\langle \omega^\natural \varphi(y),x\rangle ,\quad \text{ for all } \, x,y\in\g.$$

\begin{thm}\label{thm:sym}
Let $d_i$ be symmetric endomorphisms of $(\g,\omega)$ and set $\mathscr{ B}_i(x,y)=\omega(d_i(x),y)$ for all $i=1,2,3$.  Then,  $(d_1,d_2,d_3)$ is an  $\varepsilon$-hyper  differential operator on the pre-Lie algebra $\g$ if and only if $(\mathscr{ B}_1,\mathscr{ B}_2,\mathscr{B}_3)$ is  an  $\varepsilon$-hyper Hessian structure. 
\end{thm}
\begin{proof}
Let $(d_1,d_2,d_3)$ be an  $\varepsilon$-hyper  differential operator on the pre-Lie algebra $\g$. Since $d_i$ for all $i=1,2,3$ are invertible symmetric endomorphisms, all $\mathscr{B}_i$ are non-degenerate and symmetric. Since $d_i$ are derivations on the pre-Lie algebra $\g$, by invariance of $\omega$,  we have
\begin{eqnarray*}
&&	\mathscr{B}_i(x\cdot y,z)-\mathscr{B}_i(x,y\cdot z)-\mathscr{B}_i(y\cdot x,z)+\mathscr{B}_i(y,x\cdot z)\\
&=&\omega(d_i[x,y],z)-\omega(d_i(x),y\cdot z)+\omega(d_i(y),x\cdot z)\\
&=&\omega(d_i[x,y],z)+\langle L_y^* \omega^\natural(d_i(x)),z\rangle -\langle L_x^* \omega^\natural(d_i(y)),z\rangle \\
&=&\omega(d_i[x,y]-[d_i(x),y]-[x,d_i(y)],z)=0,
\end{eqnarray*}
which implies that $\mathscr{B}_i$ are Hessian structures on the pre-Lie algebra $\g$.  Also, we have 
$N_i=(\mathscr{B}_{i-1}^\natural)^{-1}\circ \mathscr{B}_{i+1}^\natural=d_{i-1}^{-1}\circ d_{i+1}.$
Thus $(\mathscr{B}_1,\mathscr{B}_2,\mathscr{B}_3)$ is  an  $\varepsilon$-hyper Hessian structure. 

Conversely, since $N_i=(\mathscr{B}_{i-1}^\natural)^{-1}\circ \mathscr{B}_{i+1}^\natural=d_{i-1}^{-1}\circ d_{i+1}$, we only need to check that $d_i$ for all $i=1,2,3$ are derivations on the pre-Lie algebra $\g$. By the fact that $\mathscr{B}_i$ are Hessian structures on the pre-Lie algebra $\g$, we have shown that $d_i$ are the derivations on the Lie algebra $\g^c$. Furthermore, by invariance of $\omega$, we have 
\begin{eqnarray*}
 && \omega(d_i(x\cdot y)-d_i(x)\cdot y-x\cdot d_i(y),z) \\
 &=&\omega(d_i(z),x\cdot y)-\omega(d_i(x)\cdot y,z) -\omega(x\cdot d_i(y),z)\\
 &=&-\omega([x,d_i(z)], y)-\omega([d_i(x),z],y)+\omega(d_i(y),[x,z])\\
 &=&-\omega([x,d_i(z)]+[d_i(x),z]-d_i([x,z]), y)=0,
\end{eqnarray*}
which implies that $d_i$ for all $i=1,2,3$ are derivations on the pre-Lie algebra $\g$.
\end{proof}

By Theorem \ref{thm:sym} and Proposition \ref{pro:hyper-hes}, we have 
\begin{cor}
		Let $d_i$ be symmetric endomorphisms of $(\g,\omega)$ for all $i=1,2,3$. Then 
			\begin{itemize}
			\item[\rm (i)]$(d_1,d_2,d_3)$ is an  $\varepsilon$-hyper  differential operator on the pre-Lie algebra $\g$ with $\varepsilon=(-1,-1,-1)$ if and only if the quadruple $(\omega,I_1,I_2,I_3)$ is a hyper anti-K\"ahler structure on the pre-Lie algebra $\g$, where $\omega^\sharp=d_3\circ d_1^{-1}\circ d_2$ and $I_i=(d_{i-1})^{-1}\circ d_{i+1}$ for $i\in \mathbb{Z}_3$.
			\item[\rm (ii)]$(d_1,d_2,d_3)$ is an  $\varepsilon$-hyper  differential operator on the pre-Lie algebra with  $\varepsilon=(1,1,-1)$ if and only if the quadruple $(\omega,I_1,I_2,I_3)$ is a para-hyper anti-K\"ahler structure on the pre-Lie algebra $\g$, where $\omega^\sharp=-d_3\circ d_1^{-1}\circ d_2$ and $I_i=(d_{i-1})^{-1}\circ d_{i+1}$ for $i\in \mathbb{Z}_3$.
		\end{itemize}
\end{cor}	

\noindent
{\bf Acknowledgments.} 
JL was  supported by the NSFC (W2412041,12371029) and the National Key Research and Development Program of China 
(2021YFA1002000). SB was supported by the grant: NYUAD-065.

\end{document}